\documentclass{amsart}
\usepackage{amsfonts, mathrsfs, amsthm, amsmath, amssymb, verbatim}

\providecommand{\R}{\mathbb{R}}

\providecommand{\B}{\mathbb{B}}
\providecommand{\avg}[1]{\langle #1 \rangle}
\providecommand{\esup}{\mathrm{ess\, sup}}
\providecommand{\str}{s\text{-}}

\title{The sharp $A_p$ constant for weights in a reverse-H\"older class}
\author{Martin Dindo\v{s} and Treven Wall}

\begin{document}
\newtheorem{thm}{Theorem}
\newtheorem{lem}{Lemma}

\begin{abstract}
Coifman and Fefferman established that the class of Muckenhoupt weights is equivalent to the class of weights satisfying the  ``reverse H\"older inequality". In a recent paper V. Vasyunin \cite{vasyunin} presented a proof of the reverse H\"older inequality with \emph{sharp} constants for the weights satisfying the usual Muckenhoupt condition. In this paper we present the inverse, that is, we use the Bellman function technique to find the sharp $A_p$ constants for weights in a reverse-H\"older class on an interval; we also find the sharp constants for the higher-integrability result of Gehring \cite{gehring}.

Additionally, we find sharp bounds for the $A_p$ constants of reverse-H\"older-class weights defined on rectangles in $\R^n$, as well as bounds on the $A_p$ constants for reverse-H\"older weights defined on cubes in $\R^n$, without claiming the sharpness.
\end{abstract}

\maketitle

\section{Introduction}
A weight $w$ (a non-negative, measurable function) on an interval $I$ is an $A_p(I)$ (or ``Muckenhoupt'') weight ($1<p<\infty$) if there is a constant $C < \infty$ such that the following inequality holds for every sub-interval $J \subset I$:
\begin{equation}\label{Apcond}
\avg{w}_J\avg{w^{1-p'}}^{p-1}_J \le C.
\end{equation}
Here $\avg{w}_I$ denotes $\frac{1}{|I|} \int_I w(t) \, dt$, the average of $w$ over $I$, and $p'$ is the conjugate exponent to $p$ ($p'=\frac{p}{p-1}$).

The reverse-H\"older (or ``Gehring") class $RH_p(I)$%
\footnote[2]{There is no standard notation for this class. Some authors use the notation $G_p$, e.g., \cite{basile} and \cite{malak}; others denote it by $B_p$, e.g., \cite{fkp} and \cite{rios}; our notation $RH_p$ follows that of \cite{cruz} and \cite{neug} and is, we feel, the most natural.} %
($1<p<\infty$) consists of all weights $w$ for which there is a constant $C < \infty$ so that following inequality holds for every sub-interval $J \subset I$:
\begin{equation}\label{RHpdef}
\frac{\avg{w^p}^{1/p}_J}{\avg{w}_J} \le C.
\end{equation}

If $w \in RH_p(I)$, then its $RH_p$ constant, denoted by $RH_p(w)$, is defined to be the smallest constant $C$ so that \eqref{RHpdef} holds for all $J \subset I$. We use $RH_p^{\delta}(I)$ to denote the class of all weights $w \in RH_p(I)$ such that $RH_p(w) \le \delta$. Note that by H\"older's inequality, the ratio \eqref{RHpdef} is never less than $1$; hence we only consider $\delta \ge 1$.

We can also define $RH_{\infty}(I)$ by taking the limit as $p \to \infty$ of \eqref{RHpdef}. Then, similarly, we say a weight $w$ is in $RH_{\infty}^{\delta}(I)$ if for every sub-interval $J \subset I$,
\begin{equation}\label{RHinftydef}
\frac{\esup_J w}{\avg{w}_J} \le \delta.
\end{equation}
It is worth noting that $RH_{\infty}(I)$ is strictly contained within $\bigcap_p RH_p(I)$. Among the $RH_p$ classes, $RH_{\infty}$ plays a role analogous to that of $A_1$ in the $A_p$ classes. Several equivalent definitions of $RH_{\infty}$ can be found in \cite{cruz}.

The class $A_p$ was first described by Muckenhoupt \cite{muckenhoupt}, and its connection with the reverse-H\"older inequality was first explored by Coifman and C. Fefferman \cite{coifman}, who established that $\bigcup_p RH_p(I) = \bigcup_p A_p(I)$; this union is called $A_{\infty}(I)$. There is an alternative description of $A_{\infty}$ weights as follows (see \cite{hruscev}); a weight $w$ is in $A_{\infty}(I)$ if there is a constant $C$ such that for all subintervals $J \subset I$, the following holds:
\begin{equation}\label{Ainftycond}
\avg{w}_J \exp(-\avg{\log(w)}_J) \le C.
\end{equation}

Our chief goal in this paper is to find the sharp constant $\tilde{C}$, depending only on $p, q$ and $\delta$ such that any $w \in RH_p^{\delta}$ satisfies \eqref{Apcond} (or \eqref{Ainftycond}) with constant $C = \tilde{C}$. We will denote the class of weights satisfying \eqref{Apcond} by $A_p^{C}(I)$  and those satisfying \eqref{Ainftycond} by $A_\infty^{C}(I)$, respectively. This result is the reverse direction of Vasyunin's work \cite{vasyunin}. He found the sharp constant $C$ such that any $w \in A_p^{\delta}$ belongs to the class $RH_q^{C}$. As a byproduct of our work on the above problem, we also are able to find the sharp constant $C$ in the embedding of $RH_p^{\delta}$ into $RH_t^C$ for $t > p$.

Our motivation to look at this problem arose when we attempted to establish a perturbation result for a certain class of nondivergence type elliptic operators. While studying this problem, we realized we needed to know to what $A_q$ class a certain elliptic measure belongs, given that we know it satisfies the reverse H\"older inequality with a known constant.

We are very grateful to Sasha Volberg who provided valuable insight and brought our attention to his own results in the diadic setting \cite{volberg} as well as results of Vasyunin \cite{vasyunin}.

For our work on the perturbation problem for nondivergence elliptic operators we must to establish these results not only on a real interval, but also on $\R^n$. This is the reason we have included the results in higher dimensions, despite the fact that in the cube case the constants are not sharp. The main point of Theorem \ref{ndimthm} is the asymptotic as $\delta\to 1$. We prove that for fixed $p$ and $q$, on a cube $Q \subset \R^n$,
\[RH_p^\delta(Q)\subset A_q^C(Q)\]
for some $C=C(\delta,p,q,n)$ and that $C\to 1$ as $\delta\to 1$.

\section{Statement of Principal Results}
To state our main results, we need to define the critical value of $q$. This value, $q^*=q^*(p,\delta)$, is the unique solution greater than one to \begin{equation}\label{qstardef}
\frac{(x/\delta)^p - 1}{x-1} = p.
\end{equation}

It is fairly easy to see that:
\begin{itemize}
\item For every $1<p<\infty$, $q^*(p,\delta) > \delta$.
\item For a fixed $\delta \ge 1$, $\lim_{p \to \infty} q^*(p,\delta) = \delta.$
\item For a fixed $1 < p < \infty$, $q^*(p, \delta) \sim \frac{(p\delta)^{p'}}{p}$ as $\delta \to \infty$.
\item For a fixed $1 < p < \infty$, $q^*(p, \delta) \sim (\delta^{p'}-1)^{1/p'}+1$ as $\delta \to 1$.
\end{itemize}
\begin{thm}\label{mainthm} For any weight $w \in RH_p^{\delta}(I)$, $1 < p <
\infty$ we have that $w\in A_q^{C_q(p,\delta)}$, i.e.,
\begin{equation}
\sup_{J \subset I}\,\avg{w}_J\avg{w^{1-q'}}^{q-1}_J \le
C_q(p,\delta)
\end{equation}
holds, where
\[
C_q(p,\delta)=
\begin{cases}
+\infty & 1< q \le q^*(p,\delta) \\
\frac{1}{q^*}\left(\frac{q-1}{q-q^*}\right)^{q-1} & q > q^*(p,\delta) \\
\end{cases},
\]
and $w\in A_\infty^{C_\infty(p,\delta)}$, i.e.,
\begin{equation}
\sup_{J \subset I}\, \avg{w}_J\exp(-\avg{\log(w)}_J) \le C_{\infty}(p,\delta),
\end{equation}
where
\[
C_{\infty}(p,\delta)= \frac{1}{q^*}\exp(q^*-1).
\]
Here, as throughout the paper, $q^* = q^*(p,\delta)$ is the solution to \eqref{qstardef}, above. The constants $C_q(p,\delta)$ and $C_{\infty}(p,\delta)$ in this statement are the best possible.
\end{thm}

\begin{thm}\label{mainthm2} If $w \in RH_{\infty}^{\delta}(I)$, then $w\in A_q^{C_q(\infty,\delta)}$, i.e.,
\begin{equation}
\sup_{J \subset I}\,\avg{w}_J\avg{w^{1-q'}}^{q-1}_J \le C_q(\infty,\delta),
\end{equation}
where
\[
C_q(\infty,\delta)=
\begin{cases}
+\infty & 1 < q \le \delta \\
\frac{1}{\delta} \left(\frac{q-1}{q-\delta}\right)^{q-1} & q > \delta \\
\end{cases},
\]
and $w\in A_\infty^{C_\infty(\infty,\delta)}$, i.e.,
\begin{equation}
\sup_{J \subset I}\, \avg{w}_J\exp(-\avg{\log(w)}_J) \le C_{\infty}(\infty,\delta),
\end{equation}
where
\[
C_{\infty}(\infty,\delta)= \frac{1}{\delta}\exp(\delta -1).
\]
Again, the constants $C_q(\infty,\delta)$ and $C_{\infty}(\infty,\delta)$ in this statement are the best possible.
\end{thm}

Since, for a fixed $\delta \ge 1$, $\lim_{p \to \infty} q^*(p,\delta) = \delta$, Theorem \ref{mainthm2} comes as no surprise considering Theorem \ref{mainthm}. However, the proof of Theorem \ref{mainthm} must be adjusted to prove Theorem \ref{mainthm2}. We will primarily address the proof of Theorem \ref{mainthm}, treating the proof of Theorem \ref{mainthm2} as a special case where the need arises.

For the other endpoints $p=1$ and $q=1$, a few comments are in order. A moment's thought reveals that $RH_1(I)$ is not an interesting class to consider, as every positive $L^1$ function on $I$ satisfies the condition. It is also evident from Theorem \ref{mainthm} that given any $\delta > 1$, and any $1 < p \le \infty$ there is a weight $w \in RH_p^{\delta}(I)$ which is excluded from at least one $A_q(I)$ class; hence, since $A_1(I) \subset \bigcap_q A_q(I)$, there is no $A_1$ constant which can represent the entire class $RH_p^{\delta}(I)$.

Our method also allows us to find the sharp constants in Gehring's self-\\
improvement result for the reverse-H\"older class \cite{gehring}. We define the critical exponent $t^* = t^*(p, \delta)$ as the unique solution greater than $p$ to $\left(\frac{\delta x}{x-1}\right)^p \frac{x-p}{x} = 1$.
\begin{thm}\label{gehringthm}
For any weight $w \in RH_p^{\delta}(I)$, $1 < p < \infty$ we have that $w\in RH_t^{C_t(p,\delta)}$, i.e.,
\begin{equation}
\sup_{J \subset I}\,\frac{\avg{w^t}^{1/t}_J}{\avg{w}_J} \le C_t(p,\delta)
\end{equation}
holds, where
\[
C_t(p,\delta)=
\begin{cases}
\frac{t^*-1}{t^*}\left(\frac{t^*}{t^*-t}\right)^{1/t} & p \le t < t^*(p,\delta) \\
+\infty & t \ge t^*(p,\delta) \\
\end{cases},
\]
and the constant $C_t(p,\delta)$ is sharp.
\end{thm}

In considering the $n$-dimensional analog of our results, we are no longer able to find the sharp constants. However, we find useful asymptotic information as $\delta \to 1$,
\begin{thm}\label{ndimthm}
Let $n > 1$ and let $Q \subset \R^n$ be a cube with sides parallel to the coordinate axes. Fix $p > 1$, $q > 1$ and $\eta > 1$. Then there is a $\delta > 1$ such that any weight $w \in RH_p^{\delta}(Q)$ is in $A_q^{\eta}(Q)$, that is,
\begin{equation}
\avg{w}_K\avg{w^{1-q'}}^{q-1}_K \le \eta,
\end{equation}
for every cube $K \subset Q$ with sides parallel to the coordinate axes.
\end{thm}

The standard definition of $RH_p$ in $n$-dimensions is based on cubes (or balls). However, if one strengthens the definition of $RH_p(I)$ to require the inequality $\frac{\avg{w^p}^{1/p}_J}{\avg{w}_J} \le C$ to hold for all bounded, open rectangles $J \subset I$, a new, smaller class of weights is formed (these classes are considered in, e.g., \cite{basile2}, \cite{kinnunen}). We will call this class \emph{strong} $RH_p(I)$ and denote it by $\str RH_p(I)$. Similarly, one can define \emph{strong} $A_q$ (denoted $\str A_q$). With these definitions in mind,
\begin{thm}\label{strongRHpthm}
Let $1 < p < \infty$, let $I \subset \R^n$ be a bounded, open rectangle, and assume $w \in \str RH_p^{\delta}(I)$. Then $w\in \str A_q^{C_q(p,\delta)}(I)$ and $w \in \str A_{\infty}^{C_{\infty}(p, \delta)}(I)$, where $C_q(p, \delta)$ and $C_{\infty}(p, \delta)$ are the constants in Theorem \ref{mainthm}, independent of $n$. Also, $\str RH_{\infty}^{\delta}(I) \subset \str A_q^{C_q(\infty,\delta)}(I)$ and $\str RH_{\infty}^{\delta}(I) \subset \str A_{\infty}^{C_{\infty}(\infty, \delta)}(I)$, where $C_q(\infty, \delta)$ and $C_{\infty}(\infty, \delta)$ are the constants in Theorem \ref{mainthm2}. Finally, $\str RH_p^{\delta}(I) \subset \str RH_t^{C_t(p,\delta)}(I)$, where $C_t(p,\delta)$ is the constant in Theorem \ref{gehringthm}. In all cases, these constants are sharp.
\end{thm}

The literature on $A_p$ and $RH_p$ weights is far too extensive to comprehensively cover here, but a small review is in order. The papers \cite{muckenhoupt} and \cite{coifman}, mentioned earlier, contain foundational results on these weights. Both \cite{rdf} and \cite{stein} are good references; they emphasize the connection to singular integral operators. There are several factorization results relating $RH_p$ and $A_q$ (see, e.g., \cite{cruz}), and C. J. Neugebauer, in \cite{neug}, uses these to prove that if $w \in RH_{\infty}^{\delta}$, then $w \in A_q$ for all $q > \delta$. Additionally, he provides conditions for weights in $RH_p^{\delta}$ to be in $A_q$. However, the results there depend on specific factorizations of the weights and the $A_q$ constants aren't provided. In \cite{malak}, the one-dimensional embedding of $RH_p$ into $A_q$ with the best range of $q$ is proven using rearrangements, but, again, this method doesn't find the $A_q$ constants. In \cite{basile}, the $RH_{\infty}$ embedding result, Theorem \ref{mainthm2}, is found using rearrangements; we include it here because it follows with little extra work from the proof of our $RH_p$ embedding result (Theorem \ref{mainthm}). This same group of authors finds, in \cite{basile2}, the embedding of \emph{strong} $RH_{\infty}$ into \emph{strong} $A_p$ in $n$ dimensions with the same constant. In improving upon Gehring's original result, \cite{gehring}, Korenovskii \cite{korenovskii} found the sharp upper bound on $t$ in the embedding $RH_p^{\delta} \subset RH_t$ in one dimension and Kinnunen \cite{kinnunen} found the same upper bound on for the \emph{strong} $RH_p$ classes in $n$ dimensions; however, neither of these methods provide the $RH_t$ constant of the embedded weight. Using the Bellman function technique, we are able to provide the sharp constants in one dimension and new results in $n$ dimensions, including sharp constants for all \emph{strong} $RH_p$ embeddings (the technique is explained in \cite{volberg}, especially in the context of classical analysis problems). However, finding the sharp constants for the usual $RH_p$ classes in $n$ dimensions remains an open problem.

The paper will proceed as follows: first, we describe the setup for the Bellman function technique in section \ref{bellmansection}. We use $\B(x)$ to denote the Bellman function. Then, in section \ref{thmproofsection} we prove Theorems \ref{mainthm}, \ref{mainthm2} and \ref{gehringthm} from an auxiliary theorem, Theorem \ref{auxthm}. In section \ref{Bguesssection} we explain the heuristics behind our ``guess'' at the explicit formula for the Bellman function; we call this guess $B$. We then show that our guess is correct, proving Theorem \ref{auxthm}, by verifying that $B(x) \le \B(x)$ (Lemma \ref{extremalweightlemma}) and $B(x) \ge \B(x)$ on $\Omega_{\delta}$ (Lemma \ref{omegaepsilonlemma}). Proving the former inequality requires finding a weight representing each $x$ in $\Omega_{\delta}$ (section \ref{weightsection}), and proving the latter requires working with domains $\Omega_{\epsilon}$ for $\epsilon > \delta$ (see section \ref{omegaepsilonsection}). Finally, we prove Theorems \ref{ndimthm} and \ref{strongRHpthm} in section \ref{ndimsection}. Throughout the paper, we alternate between heuristic calculations and rigorous proof to exhibit the philosophy of the Bellman function technique.

\section{Bellman function ideas}\label{bellmansection}
Typically, when one uses the Bellman function technique, all one needs is to find an upper bound for the Bellman function which preserves concavity (or convexity, as needed). However, this approach doesn't allow for the calculation of sharp constants. Consequently, we find the actual formula for the Bellman function.

For all weights $w$, for $1 < p < \infty$, and for any interval $J$, $\avg{w}^p_J \le \avg{w^p}_J$, by H\"older's inequality. Hence, if $w \in RH_p^{\delta}(I)$, the point $x = (x_1, x_2) = (\avg{w}_I, \avg{w^p}_I)$ lies in the domain
\[\Omega_{\delta}(p):= \{(x_1, x_2) :  x_1 > 0, \, x_1^p \le x_2 \le (\delta x_1)^p \}.\]
For our problem, the Bellman function for $1 < p < \infty$, $0 < q < \infty$ is
\begin{equation}
\B(x; p, q, \delta):= \sup \{\avg{w^{1-q'}}_I : x_1=\avg{w}_I, x_2=\avg{w^p}_I, \text{ and } w \in RH_p^{\delta}(I) \},
\end{equation}
and for $q = \infty$,
\begin{equation}
\B(x; p, \infty, \delta) := \sup \{\exp(-\avg{\log(w)}_I): x_1 = \avg{w}_I, x_2 = \avg{w^p}_I, w \in RH_p^{\delta}(I) \}.
\end{equation}
Note that $\B$ doesn't depend on the interval $I$ on which it is defined, since, given two intervals $I_1$ and $I_2$, the affine mapping of one onto the other preserves the averages and puts $RH_p^{\delta}(I_1)$ in one-to-one correspondence with $RH_p^{\delta}(I_2)$. We are allowing for $0 < q < 1$ in order to prove Theorem \ref{gehringthm}; for $q$ in this range, the exponent $1-q'$ is greater than 1.

For $p = \infty$, we must adjust these coordinates. We set $x = (x_1, x_2) =$ \\
$(\avg{w}_I, \esup_I w)$, whence
\begin{align*}
\Omega_{\delta}(\infty) &:= \{(x_1, x_2) :  x_1 > 0, \, x_1 \le x_2 \le \delta x_1 \},\\
\B(x; \infty, q, \delta) &:= \sup \{\avg{w^{1-q'}}_I : x_1=\avg{w}_I, x_2=\esup_I w, \text{ and } w \in RH_{\infty}^{\delta}(I) \},\\
\B(x; \infty, \infty, \delta) &:= \sup \{\exp(-\avg{\log(w)}_I): x_1 = \avg{w}_I, x_2 = \esup_I w, w \in RH_{\infty}^{\delta}(I) \}.
\end{align*}

We consider $\B$ as a function on $\Omega_{\delta}$, since each point $x \in \Omega_{\delta}$ can be represented by a weight $w \in RH_p^{\delta}$. We will demonstrate the existence of such weights in Lemma \ref{extremalweightlemma}.

We will often split an interval $J$ into the union of two disjoint subintervals which we will call $J^-$ and $J^+$, with $|J^{\pm}| = \alpha^{\pm}|J|$. Given a weight $w$ defined on $J$, we split it into two weights $w^{\pm}$ defined on their respective subintervals. As above, we relate these weights to points in $\R^2$, letting the point $x^0$ correspond to the original weight $w$ and the points $x^{\pm}$ correspond to $w^{\pm}$. These points are co-linear: $x^0 = \alpha^- x^- + \alpha^+ x^+$. Also, if we start with $w \in RH_p^{\delta}(J)$, $w^{\pm}$ are in $RH_p^{\delta}(J^{\pm})$; consequently, $x^0$, $x^-$ and $x^+$ are all points in $\Omega_{\delta}$.

We need some further notation. For $1 < p < \infty$, denote by $u_p^{\pm}$ the functions inverse to
\[
t \to \frac{(1-pt)^{p-1}}{(1-(p-1)t)^p}
\]
on the following domains: $u_p^+:[0,1] \to [0,\frac{1}{p}]$, $u_p^-:[0,1] \to (-\infty, 0]$, i.e., the values $u_p^{\pm}(t)$ are the positive and negative solutions to the equation $\frac{(1-pu)^{p-1}}{(1-(p-1)u)^p} = t$ for $0 \le t \le 1$. Based on this, we define $s^{\pm} = s_p^{\pm}(\delta) := u_p^{\pm}(1/\delta^p)$ and $r^{\pm} = r^{\pm}_p(x, \delta) := u_p^{\pm}(x_2/(\delta x_1)^p)$. Finally, we set $\gamma := p + q' -1$.

Theorems \ref{mainthm}, \ref{mainthm2} and \ref{gehringthm} are consequences of the following
\begin{thm}\label{auxthm}
For $0 < q < \infty$ and $1 < p < \infty$, if $x_2 = x_1^p$ (or if $p = \infty$ and $x_2 = x_1$), we have
\[
\B(x; p, q, \delta) = x_1^{1-q'}.
\]
If $1 < p < \infty$ and $x_2 > x_1^p$, then
\begin{equation}\label{bellmanfn}
\B(x;p,q, \delta) = \qquad \qquad
\end{equation}
\[
\begin{cases}
x_1^{1-q'}\left(\frac{1-ps^+}{1-pr^+}\right)^{q'}\left(\frac{1-(p-1)r^+}{1-(p-1)s^+}\right)^{q'-1}\left(\frac{1-\gamma r^+}{1-\gamma s^+}\right) & \text{ for } q > \frac{1-(p-1)s^+}{1-ps^+} \\
\infty & \text{ for } \frac{1-(p-1)s^-}{1-ps^-} \le q \le \frac{1-(p-1)s^+}{1-ps^+}\\
x_1^{1-q'}\left(\frac{1-ps^-}{1-pr^-}\right)^{q'}\left(\frac{1-(p-1)r^-}{1-(p-1)s^-}\right)^{q'-1}\left(\frac{1-\gamma r^-}{1-\gamma s^-}\right) & \text{ for } \frac{p-1}{p} < q < \frac{1-(p-1)s^-}{1-ps^-}.\\
\end{cases}
\]
If $p = \infty$ and $x_2 > x_1$, then
\begin{equation}\label{pinftybellmanfn}
\B(x; \infty, q, \delta) =
\begin{cases}
x_2^{1-q'}\left(\frac{q-\frac{x_1}{x_2}\delta}{q-\delta}\right) & \text{ for } q > \delta \\
\infty & \text{ for } 1 < q \le \delta.\\
\end{cases}
\end{equation}
\end{thm}

\section{Proof of Theorems \ref{mainthm}, \ref{mainthm2} and \ref{gehringthm}}\label{thmproofsection}
\begin{proof} One can easily check that the value $q^*(p,\delta)$ defined by \eqref{qstardef} and used in Theorem \ref{mainthm} is the same as $\frac{1-(p-1)s^+}{1-ps^+}$ used in Theorem \ref{auxthm}. Similarly, if we define $q_* = q_*(p,\delta)$ to be the unique solution to \eqref{qstardef} between 0 and 1, $q_* = \frac{1-(p-1)s^-}{1-ps^-},$ the other bound in Theorem \ref{auxthm}. The critical exponent in Theorem \ref{gehringthm}, $t^*$, satisfies $t^* = 1 - q_*' = \frac{1}{1-q_*}$.

Now, we assume Theorem \ref{auxthm}. We consider $q > 1$ for Theorems \ref{mainthm} and \ref{mainthm2} and $\frac{p-1}{p} < q < 1$ for Theorem \ref{gehringthm} (in this range, $1-q' > p$, which is all that we are interested in). Recall that $\B(x)$ represents the maximum of $\avg{w^{1-q'}}$ for all weights in $RH_p^{\delta}$ which are represented by $x$ and that $x_1$ represents $\avg{w}$. For $q > 1$, the constant we desire is
\[
\sup_{x \in \Omega_{\delta}} x_1 \big(\B(x; p, q, \delta)\big)^{q-1},
\]
and for $\frac{p-1}{p} < q < 1$, we seek
\[
\sup_{x \in \Omega_{\delta}} \frac{1}{x_1}\B(x; p, q, \delta)^{1-q}.
\]
With that in mind, we define $g := x_1^{q'-1}\B(x)$, that is,
\[
g = \left(\frac{1-ps^{\pm}}{1-pr^{\pm}}\right)^{q'}\left(\frac{1-(p-1)r^{\pm}}{1-(p-1)s^{\pm}}\right)^{q'-1}\left(\frac{1-\gamma r^{\pm}}{1-\gamma s^{\pm}}\right).
\]
$g \ge 0$ and
\[
\frac{d}{dr^{\pm}} \log(g) = \frac{-q'(q'-1)r^{\pm}}{(1-pr^{\pm})(1-(p-1)r^{\pm})(1-\gamma r^{\pm})}.
\]
For $q > 1$, we use $r^+$ and $s^+$, so $g'$ is negative. For $\frac{p-1}{p} < q < 1$, we use $r^-$ and $s^-$, so $g'$ is positive. Hence, the maximum of $g$ in both cases is at $r^{\pm} = 0$. So, our best constant is
\begin{align*}
\left(g(0)\right)^{q-1} &= \frac{(1-ps^+)^{q}}{(1-(p-1)s^+)(1-\gamma s^+)^{q-1}}, \text{ for } q > 1 \text{, and}\\
\left(g(0)\right)^{1-q} &= \frac{1-(p-1)s^-}{(1-ps^-)^{q}(1-\gamma s^-)^{1-q}} \text{ for } \frac{p-1}{p} < q < 1.
\end{align*}
Relating this first constant back to $q^*$, we find
\begin{align*}
\frac{(1-ps^+)^{q}}{(1-(p-1)s^+)(1-\gamma s^+)^{q-1}} &= \frac{1}{q^*}\left(1 + \frac{s^+(1-q')}{1-ps^+}\right)^{1-q} \\
& = \frac{1}{q^*} \left(1 - \frac{q^* - 1}{q-1}\right)^{1-q} = \frac{1}{q^*} \left(\frac{q-1}{q - q^*}\right)^{q-1},
\end{align*}
which is the constant in Theorem \ref{mainthm}. For the second constant, $\frac{p-1}{p} < q < 1$ and we use $t = 1- q'$ and $t^* = 1-q_*'$ to see
\begin{align*}
\frac{(1-(p-1)s^-)}{(1-ps^-)^{q}(1-\gamma s^-)^{1-q}} &= q_* \left(\frac{q - q_*}{q-1}\right)^{q-1} = \frac{t^* - 1}{t^*} \left((1-t) + t(\frac{t^*-1}{t^*})\right)^{-1/t} \\
&= \frac{t^* - 1}{t^*} \left(\frac{t^*}{t^* - t}\right)^{1/t},
\end{align*}
which is the constant in Theorem \ref{gehringthm}.

To complete the proof for $p < \infty$, fix a point $x \in \Omega_{\delta}$ (fixing $r^{\pm}= r^{\pm}(x)$). Then, the weight
\begin{equation}\label{extremalweight}
w_{c,a,\nu}(t) =
\begin{cases}
c\left(\frac{t}{a}\right)^{\nu}& \text{ if } 0 \le t \le a \\
c& \text{ if } a \le t \le 1,
\end{cases}
\end{equation}
with constants $\nu = \frac{s^{\pm}}{1-ps^{\pm}},$ $a = \frac{s^{\pm}-r^{\pm}}{s^{\pm}(1-pr^{\pm})},$ $c = x_1 \frac{(1-pr^{\pm})(1-(p-1)s^{\pm})}{(1-(p-1)r^{\pm})(1-ps^{\pm})},$ is in $RH_p^{\delta}(I)$ and its $A_q$ norm is infinite for any $q_* \le q \le q^*$. This is exhibited in the proof of Lemma \ref{extremalweightlemma}.

For the case of $p = \infty$, the analysis is even easier. Given the definition of $\B$ in \eqref{pinftybellmanfn}, we see that for $q > \delta$,
\[
x_1 \big(\B(x; p, q, \delta)\big)^{q-1} = \frac{x_1}{x_2} \left(\frac{q-\frac{x_1}{x_2}\delta}{q-\delta}\right)^{q-1}.
\]
Letting $y = \frac{x_1}{x_2}$, we know that $\frac{1}{\delta} \le y \le 1$, and we see
\[
\frac{d}{dy} \left[y \left(\frac{q-y\delta}{q-\delta}\right)^{q-1}\right] =  q(1-\delta y)\frac{(q-y\delta)^{q-2}}{(q-\delta)^{q-1}},
\]
which is negative for $y > 1/\delta$. Thus, the maximum is at $y = 1/\delta$, which is exactly the constant in Theorem \ref{mainthm2}. For $q \le \delta$, we again fix an $x \in \Omega_{\delta}$. Then, the weight in \eqref{extremalweight}, with constants $\nu = \delta - 1$, $a = \frac{1-x_1/x_2}{1-1/\delta}$ and $c=x_2$, is in $RH_{\infty}^{\delta}(I)$ with infinite $A_q$ norm for $1 \le q \le \delta$, which completes the proof. As before, this is contained in the proof of Lemma \ref{extremalweightlemma}.

Finally, we address the case of $q=\infty$. Define, for $1 < p < \infty$,
\begin{align}\label{qinftybellman}
B(x; p, \infty, \delta) &:= \lim_{q \to \infty} \left(\B(x;p,q,\delta)\right)^{q-1}\\
&= \frac{1}{x_1} \frac{(1-(p-1)r^+)(1-ps^+)}{(1-pr^+)(1-(p-1)s^+)} \exp[\frac{s^+-r^+}{(1-ps^+)(1-pr^+)}] \notag,
\end{align}
and, for $p = \infty$,
\begin{equation}\label{pqinftybellman}
B(x; \infty, \infty, \delta) := \lim_{q \to \infty} \left(\B(x;\infty,q,\delta)\right)^{q-1} = \frac{1}{x_2} \exp[\delta(1-\frac{x_1}{x_2})].
\end{equation}
We want to establish that these functions satisfy
\[
B(x; p, \infty, \delta) = \B(x; p, \infty, \delta) \quad \text{and} \quad B(x; \infty, \infty, \delta) = \B(x; \infty, \infty, \delta).
\]
First, it is not difficult to check that the weights $w_p$ and $w_{\infty}$ defined by \eqref{extremalweight} with the respective constants for $p < \infty$, $p=\infty$ do not depend on $q$ and satisfy
\[
\exp(-\avg{\log(w_p)}) = B(x; p, \infty, \delta) \quad \text{ and } \quad \exp(-\avg{\log(w_{\infty})}) = B(x; \infty, \infty, \delta),
\]
respectively. By the definition of $\B$, this gives the inequalities $B(x; p, \infty, \delta) \le \B(x; p, \infty, \delta)$ and $B(x; \infty, \infty, \delta) \le \B(x; \infty, \infty, \delta)$. The other inequality is a result of applying Jensen's inequality; namely,
\[
\exp(-\avg{\log(w)}) \le \avg{w^{1-q'}}^{q-1} \le \B(x; p, q, \delta)^{q-1}.
\]
So in \eqref{qinftybellman} and \eqref{pqinftybellman}, taking the limits establishes the desired equality of $B$ and $\B$.

Given this, the $A_{\infty}$ constants in Theorems \ref{mainthm} and \ref{mainthm2} are easy to find. We simply calculate
\[
\sup_{x \in \Omega_{\delta}} x_1 \B(x; p, \infty, \delta) = \sup_{x \in \Omega_{\delta}} \frac{(1-(p-1)r^+)(1-ps^+)}{(1-pr^+)(1-(p-1)s^+)} \exp[\frac{s^+-r^+}{(1-ps^+)(1-pr^+)}].
\]
Again, the maximum is at $r^+ = 0$, whence the constant is
\[
\frac{(1-ps^+)}{(1-(p-1)s^+)} \exp[\frac{s^+}{(1-ps^+)}] = \frac{1}{q^*} \exp(q^* - 1).
\]
And,
\[
\sup_{x \in \Omega_{\delta}} x_1 \B(x; \infty, \infty, \delta) = \sup_{x \in \Omega_{\delta}} \frac{x_1}{x_2} \exp[\delta(1-\frac{x_1}{x_2})] = \frac{1}{\delta} \exp(\delta - 1).
\]
This completes the proof of Theorems \ref{mainthm}, \ref{mainthm2} and \ref{gehringthm} from Theorem \ref{auxthm}.
\end{proof}

\section{Deriving the formula for $\B$}\label{Bguesssection}
We start by examining the scaling properties of $\B$. Given $w \in RH_p^{\delta}$, and $\lambda > 0$ a constant, then $\tilde{w} := \lambda w$ is in $RH_p^{\delta}$ as well, and $\avg{\tilde{w}}_I = \lambda\avg{w}_I$, $\avg{\tilde{w}^p}_I = \lambda^p\avg{w^p}_I$. Consequently, $\B(\lambda x_1, \lambda^p x_2) = \lambda^{1-q'}\B(x_1, x_2)$. Letting $x_1 = \frac{1}{\lambda}$, we see $\B(x_1,x_2)=x_1^{1-q'}\B(1,\frac{x_2}{x_1^p})$. Thus, we define $g(y) := \B(1, y)$, and we see that $\B(x_1,x_2) = x_1^{1-q'} g(\frac{x_2}{x_1^p})$.

For $0 < q < \infty$, $\avg{w^{1-q'}}_I \ge \avg{w}^{1-q'}_I$, using H\"older's inequality, whence, \\
$\B(x_1, x_2) \ge x_1^{1-q'}$. Therefore, $g(y) \ge 1$. Further, if $x_2 = x_1^p$, that is, if $\avg{w^p}_I = \avg{w}^p_I$, the weight $w$ must be a constant. In that case, $\B(x) = x_1^{1-q'}$, and we see that $g(1)=1$.

We expect $\B$ to be a concave function, as the following illustrates. Given an interval $J$, split it into the disjoint union of subintervals $J^-$ and $J^+$. Assuming they exist, let $w^{\pm}$ be two extremal weights (i.e., which satisfy $\avg{(w^{\pm})^{1-q'}}_{J^{\pm}} = \B(x^{\pm})$). Then, concatenate these two weights to form a new weight $w$ on $J$. Thus, $\avg{w^{1-q'}}_J = \alpha^- \avg{(w^-)^{1-q'}}_{J^-} + \alpha^+ \avg{(w^+)^{1-q'}}_{J^+}$. The weight $w$ corresponds to the point $x^0$, and $w^{\pm}$ to $x^{\pm}$. Then, $x^0 = \alpha^- x^- + \alpha^+ x^+$, and we have $\B(x^0) \ge \alpha^-\B(x^-) + \alpha^+\B(x^+)$, which is the concavity condition (alternatively, using the terminology of \cite{volberg}, we expect a concave solution since the profit function is zero and there is no drift term). We ignore the substantive issues of whether extremal weights exist, and whether $x^0 \in \Omega_{\delta}$ due to the heuristic nature of our procedure. However, later proofs lay these concerns to rest. We will also assume that the Hessian of $\B$ is singular. This last assumption gives rise to an ODE that we can solve, which enables us to find $\B$ explicitly. This assumption is frequently made in the application of the Bellman function technique; here it is reasonable because we expect extremal weights to exist.

To proceed further, we must calculate the Hessian of $\B$ in terms of $g$ (assuming, of course, that $\B$ is sufficiently differentiable). Let $y = \frac{x_2}{x_1^p}$. Then,
\begin{gather*}
\frac{\partial^2 \B}{\partial x_1^2} = x_1^{-q'-1} [q'(q'-1)g + (2q'+p-1)pyg' + p^2y^2g''], \\
\frac{\partial^2 \B}{\partial x_1 \partial x_2} = x_1^{-q'-p}[(1-q'-p)g'-pyg''], \quad \text{and} \quad \frac{\partial^2 \B}{\partial x_2^2} = x_1^{1-q'}x_2^{-2}y^2g''.
\end{gather*}
Thus, \begin{equation} \text{Hess}(\B) =
\begin{pmatrix}
\frac{\partial^2 \B}{\partial x_1^2} & \frac{\partial^2 \B}{\partial x_1 \partial x_2} \\
\frac{\partial^2 \B}{\partial x_2 \partial x_1} & \frac{\partial^2 \B}{\partial x_2^2}
\end{pmatrix}
= x_1^{-q'-1} \begin{pmatrix} 1 & 1 \\ 0 & -\frac{x_1}{px_2}\\ \end{pmatrix}
R \begin{pmatrix} 1 & 0 \\ 1 & -\frac{x_1}{px_2} \\ \end{pmatrix}
\end{equation}
with
\begin{equation}
R = \begin{pmatrix} q'(q'-1)g-p(p-1)yg' & p\gamma yg' \\ p\gamma yg' & p^2y^2g'' \\ \end{pmatrix},
\end{equation}
where $\gamma = p+q'-1$.

To force $R$ to be singular, we require
\begin{equation}
(p^2y^2g'')(q'(q'-1)g - p(p-1)yg') = (p\gamma y g')^2.
\end{equation}
Now, we make the substitution $h = \frac{yg'}{g}$. Then $g'' = \frac{gh' + hg' - g'}{y}$. Convert and divide both sides by $(pg)^2$ to get
\[(yh' - h + h^2)(q'(q'-1) - p(p-1)h) = \gamma^2 h^2,\] which is separable. 
So, \[ y = C \left( \frac{h((p-1)h+q')^{p-1}}{(ph+q'-1)^p} \right).\]
We make one further change of variables, \[h = \frac{q'(q'-1)}{p(p-1)(1-\gamma r)} \qquad \text{ or } \qquad r = \frac{1}{\gamma} - \frac{q'(q'-1)}{\gamma p(p-1)h},\]
which yields
\begin{equation}\label{upeqn}
y = C \frac{(1-pr)^{p-1}}{(1-(p-1)r)^p}.
\end{equation}
It turns out to be natural to choose $C = \delta^p$ in \eqref{upeqn}, and it is at this point that we see the origin of the function $u_p$, mentioned above. Recall that we have set $u_p^{\pm}$ as the positive and negative inverses of the function
\[ t \to \frac{(1-pt)^{p-1}}{(1-(p-1)t)^p}. \] Also recall that $s^{\pm} = u_p^{\pm}(\frac{1}{\delta^p})$. We can see from \eqref{upeqn}, with $C = \delta^p$, that $r^{\pm} = u_p^{\pm}(\frac{x_2}{(\delta x_1)^p})$. Note also that if $r = s$, then $y = 1$.

We want to relate this all back to $g$, so we calculate
\[ \frac{d}{dr}\log(y) = \frac{-p(p-1)}{1-pr} + \frac{p(p-1)}{1-(p-1)r}\] and use this to see that
\begin{align*}
d(\log(g)) &= h d(\log(y)) = \frac{q'(q'-1)}{p(p-1)(1-\gamma r)} \left(\frac{-p(p-1)}{1-pr} + \frac{p(p-1)}{1-(p-1)r} \right) \\
&= \frac{pq'}{1-pr} - \frac{\gamma}{1-\gamma r} - \frac{(p-1)(q'-1)}{1-(p-1)r}.
\end{align*}

Since $g |_{y=1} = g|_{r=s} = 1$, we have
\begin{align*}
\log g &= \int_s^r \left( \frac{pq'}{1-pt} - \frac{\gamma}{1-\gamma t} - \frac{(p-1)(q'-1)}{1-(p-1)t} \right) \, dt \\
&= -q'\log(\frac{1-pr}{1-ps}) + (q'-1)\log(\frac{1-(p-1)r}{1-(p-1)s}) + \log(\frac{1-\gamma r}{1-\gamma s}),
\end{align*}
whence
\begin{equation}
g = \left(\frac{1-ps}{1-pr}\right)^{q'}\left(\frac{1-(p-1)r}{1-(p-1)s}\right)^{q'-1}\left(\frac{1-\gamma r}{1-\gamma s}\right).
\end{equation}
The last thing is to discover whether we should use $r^+$ or $r^-$ in the definition of $g$ to ensure that $R$ is negative semi-definite. Since $R$ is singular and symmetric, it suffices to make the upper left-hand entry of $R$ negative. That is, we must make sure that
\[ p(p-1)y g' \ge q'(q'-1)g \quad \text{ or } \quad h \ge \frac{q'(q'-1)}{p(p-1)}. \]
Note that \[ \frac{dh}{dr} = \frac{q'(q'-1)}{p(p-1)} \gamma (1- \gamma r)^2,\]
which is positive for $q > 1$ and negative for $\frac{p-1}{p} < q < 1$. Also, $h(0) = \frac{q'(q'-1)}{p(p-1)}$, so we need $h(r) \ge h(0)$ for $q > 1$, which is accomplished by choosing the positive solution $r^+$. Accordingly, for $\frac{p-1}{p} < q < 1$, we use $r^-$.

Therefore, we have a candidate for $\B$: $B(x) = x_1^{1-q'}g(\frac{x_2}{x_1^p})$. It is occasionally helpful to have this expressed in two different ways, which we record here
\begin{align}
B(x) &= x_1^{1-q'}\left(\frac{1-ps^{\pm}}{1-pr^{\pm}}\right)^{q'}\left(\frac{1-(p-1)r^{\pm}}{1-(p-1)s^{\pm}}\right)^{q'-1}\left(\frac{1-\gamma r^{\pm}}{1-\gamma s^{\pm}}\right) \label{B1} \\
B(x) &= x_1^{-\gamma}x_2\left(\frac{1-ps^{\pm}}{1-pr^{\pm}}\right)^{\gamma}\left(\frac{1-(p-1)r^{\pm}}{1-(p-1)s^{\pm}}\right)^{\gamma}\left(\frac{1-\gamma r^{\pm}}{1-\gamma s^{\pm}}\right) \label{B2}.
\end{align}
The second representation is obtained by using the definitions of $r^{\pm}$ and $s^{\pm}$ to see that
\[\frac{x_2}{x_1^p} = \left(\frac{1-pr^{\pm}}{1-ps^{\pm}}\right)^{p-1}\left(\frac{1-(p-1)s^{\pm}}{1-(p-1)r^{\pm}}\right)^{p}.\]
We note that since $s^- \le r^- \le 0 \le r^+ \le s^+ < 1/p$, the only concern with the denominator of $g$ occurs when $(1- \gamma s^{\pm}) =0$. At this point, we have $q' = 1/s^{\pm} + 1 - p$, or $q = \frac{1-(p-1)s^{\pm}}{1-ps^{\pm}}$, which is $q^*$ (or $q_*$), the critical values of $q$ in Theorems \ref{mainthm} and \ref{auxthm}.

For the case of $p=\infty$, we have two options. Either we can take limits, using the asymptotics as $p \to \infty$ of
\begin{equation}\label{asymptotics}
s^+ \sim \frac{1}{p} - \frac{1}{(\delta - 1)p^2} \quad \textrm{and} \quad r^+ \sim \frac{1}{p} - \frac{1}{(\delta\tilde{x}-1)p^2},
\end{equation}
where $\tilde{x} := \frac{x_1}{x_2^{1/p}} \to \frac{x_1}{x_2}$. Or, we can carry out a similar analysis. We leave the former approach to the reader and illustrate the latter approach, as the ideas involved are useful for later. We start by recalling that for a weight $w$ and an interval $J$, $x_2 = \esup_{J} w$. This change of coordinates alters the effect of splitting; if we split an interval $J$ into two parts, $J = J^- \cup J^+$, the point $x^0 = (\avg{w}_J, \esup_J w)$ is no longer necessarily co-linear with the points $x^{\pm} = (\avg{w}_{J^{\pm}}, \esup_{J^{\pm}} w)$. The first coordinate splits proportionally, $x_1^0 = \alpha^- x^-_1 + \alpha^+ x^+_1$, but the second coordinate now satisfies $x^0_2 = \textrm{max}\{x^-_2, x^+_2\}$. Nevertheless, we are still seeking a concavity condition as before, that is, we want to ensure that $\B(x^0) \ge \alpha^- \B(x^-) + \alpha^+ \B(x^+).$ However, this is not typical concavity, due to the behavior of the coordinates. To get our hands on an expression for the concavity, we look at the Taylor series for $\B$ based at $x^0$ up to the second terms (assuming $\B$ is sufficiently differentiable):
\begin{equation}
\B(x^{\pm}) \simeq \B(x^0) + \sum_{i=1}^2 \frac{\partial \B}{\partial x_i}(x^0)(x^{\pm}_i-x^0_i) + \frac{1}{2} \sum_{i,j=1}^2 \frac{\partial^2\B}{\partial x_i \partial x_j}(x^0)(x^{\pm}_i - x^0_i)(x^{\pm}_j - x^0_j).
\end{equation}
We know one of $x^{\pm}_2$ is equal to $x^0_2$, so we assume that $x^+_2 = x^0_2$; due to the symmetry of the situation, we lose no generality in this assumption. Recall that $x^0_1 = \alpha^-x^-_1 + \alpha^+x^+_1$, and define $\Delta_1 := x^+_1 - x^-_1$ and $\Delta_2 := x^0_2 - x^-_2 \ge 0$. Then, for concavity, we want the following linear combination of these terms to be non-positive for small $\Delta_1$ and $\Delta_2$:
\begin{align}\label{strangeconcavity}
\alpha^-\B(x^-) &+ \alpha^+\B(x^+) - \B(x^0) \simeq \notag \\
& \alpha^- \left( -\frac{\partial \B}{\partial x_2}\Delta_2 + \frac{1}{2}\alpha^+ \frac{\partial^2 \B}{\partial x_1^2}\Delta_1^2 + \alpha^+ \frac{\partial^2\B}{\partial x_1 \partial x_2}\Delta_1\Delta_2 + \frac{1}{2} \frac{\partial^2\B}{\partial x_2^2}\Delta_2^2 \right).
\end{align}

Assume that Hess$(\B)$ is singular, that $\frac{\partial \B}{\partial x_2} \ge 0$ and that $\frac{\partial^2 \B}{\partial x_1^2} \le 0$; we will then demonstrate that \eqref{strangeconcavity} is non-positive. From the assumptions that Hess$(\B)$ is singular and $\frac{\partial^2 \B}{\partial x_1^2} \le 0$, we know that $\frac{\partial^2 \B}{\partial x_2^2} \le 0$ and that the quadratic form of Hess$(\B)$ is negative semi-definite, hence
\[\frac{\partial^2 \B}{\partial x_1 \partial x_2} \Delta_1 \Delta_2 \le -\frac{1}{2}\left( \frac{\partial^2 B}{\partial x_1^2}\Delta_1^2 + \frac{\partial^2 \B}{\partial x_2^2}\Delta_2^2\right).\]
Therefore, the right-hand side of \eqref{strangeconcavity} is less than or equal to
\[ \alpha^- \left(- \frac{\partial \B}{\partial x_2}\Delta_2 + \frac{1}{2} (1- \alpha^+)\frac{\partial^2 \B}{\partial x_2^2}\Delta_2^2\right),\] which is non-positive.
To get the differential equation which defines $B$, we supplement these conditions with yet another singularity assumption and arrive at two possibilities
\begin{align}
\frac{\partial^2 \B}{\partial x_1^2} = 0 &\qquad \text{and} \qquad \frac{\partial \B}{\partial x_2} \ge 0, \quad \text{or} \label{de1}\\
\frac{\partial \B}{\partial x_2} = 0 &\qquad \text{and} \qquad \frac{\partial^2 \B}{\partial x_1^2} \le 0. \label{de2}
\end{align}

We then turn to the scaling of $\B$. In this case, we see $\B(\lambda x_1, \lambda x_2) =$ \\
$\lambda^{1-q'}\B(x_1,x_2)$, so setting $\lambda = 1/x_1$ yields $\B(x_1,x_2) = x_1^{1-q'}\B(1,\frac{x_2}{x_1})$. We then define $g(y) = \B(1,y)$, so that $\B(x_1,x_2) = x_1^{1-q'}g(\frac{x_2}{x_1})$. Also, as before, we know that $\B(x) \ge x_1^{1-q'}$, whence $g(y) \ge 1$. Further, $\frac{x_2}{x_1} = 1$ only occurs when the weight is constant, in which case we have $\B(x_1, x_1; q, \delta) = x_1^{1-q'}$; that is, $g(1)=1$.

If $\B$ solves \eqref{de1}, then $\B$ is linear in $x_1$, which yields
\[\B(x_1, x_2) = a(x_2) + b(x_2)x_1 = x_1^{1-q'}g(x_2/x_1).\] Since $g(1)=1$, we set $x_1=x_2$ and find
\[a(x_2) = x_2^{1-q'} - b(x_2)x_2.\] Then substitute this in and multiply both sides by $x_2^{q'-1}$ to get
\[1 + x_2^{q'}b(x_2)\left(\frac{x_2}{x_1}-1\right) = \left(\frac{x_2}{x_1}\right)^{q'-1} g(\frac{x_2}{x_1}). \text{ So,}\]
\[x_2^{q'}b(x_2) = \frac{(\frac{x_2}{x_1})^{q'-1}g(\frac{x_2}{x_1}) - 1}{\frac{x_1}{x_2} - 1} = c, \] a constant. Then,
\begin{equation}\label{Bformula}
\B(x) = x_2^{1-q'}\left(1 - c + c \frac{x_1}{x_2}\right).
\end{equation}
If we let $y = \frac{x_2}{x_1}$, then $y \in [1,\delta]$ and we see $g(y) = y^{-q'}[(1-c)y + c]$. Next, \[0 \le g(y) - 1 = y^{-q'}(y-1)\left( \frac{y- y^{q'}}{y - 1} - c \right),\] and since the function $y \to \frac{y-y^{q'}}{y-1}$ is monotone decreasing, we see that
\begin{equation}\label{cbound}
c \le \frac{\delta - \delta^{q'}}{\delta - 1} < 1 - q'.
\end{equation}
Now, we calculate
\[\frac{\partial \B}{\partial x_2} = x_2^{1-q'}\left[(1-c)(1-q') - q'c\frac{x_1}{x_2}\right],\] which needs to be non-negative to satisfy \eqref{de1}. At $y=1$, the expression above is positive, given \eqref{cbound}. If $q \le \delta$, there is no possible value for $c$, because then $q' \ge \frac{\delta}{\delta -1}$, in which case $(1-c)(1-q') - \frac{q'c}{\delta} = c\left(q'- 1 - q'/\delta\right) + 1-q' \ge 0$ contradicts \eqref{cbound}. So, assuming $q > \delta$, we want to choose $c$ in such a way that $\left[(1-c)(1-q') - q'c\frac{x_1}{x_2}\right]$ stays positive on the entire interval $[1,\delta]$. Solving for $c$, we get
\[c = \frac{\delta(q'-1)}{\delta(q'-1) - q'} = \frac{\delta}{\delta - q}.\] Substituting this into \eqref{Bformula} yields our candidate function for $\B(x; \infty, q, \delta)$, \[B(x; \infty, q, \delta) = x_2^{1-q'}\left(\frac{q-\frac{x_1}{x_2}\delta}{q-\delta}\right),\] which is what appears in \eqref{pinftybellmanfn}.

If $\B$ solves \eqref{de2}, then $\B$ is constant in $x_2$, in which case $g$ must be constant, whence $g \equiv 1$ and $\B = x_1^{1-q'}$. But, then $\frac{\partial^2 \B}{\partial x_1^2} = (1-q')(-q')x_1^{-q'-1}$, which is positive. Therefore \eqref{de2} has no solution for $q > 1$.

\section{$B$ is concave}
We now proceed to verify, in several steps, that $B = \B$, proving Theorem \ref{auxthm}. Our first lemma addresses the fact that $B$ is, indeed, concave.

\begin{lem}\label{concavitylemma}
\textbf{Case 1, $p < \infty$:} Let $x^{\pm}$ be two arbitrary points in $\Omega_{\delta}$. If the entire line segment joining these two points (denoted $[x^-, x^+]$) is contained within $\Omega_{\delta}$, then
\begin{equation}
B(\alpha^- x^- + \alpha^+ x^+) \ge \alpha^- B(x^-) + \alpha^+ B(x^+)
\end{equation}
holds for all non-negative numbers $\alpha^{\pm}$ with $\alpha^- + \alpha^+ = 1$.

\textbf{Case 2, $p = \infty$:} Let $x^{\pm}$ be two arbitrary points in $\Omega_{\delta}$ and let $\alpha^{\pm}$ be a pair of non-negative numbers such that $\alpha^- + \alpha^+ = 1$. Define $x^0 = (x^0_1, x^0_2) := (\alpha^- x_1^- + \alpha^+ x_1^+, \text{max}\{x_2^-, x_2^+\})$. If both of the points $(x_1^{\pm}, x^0_2)$ are in $\Omega_{\delta}$, then
\begin{equation}
B(x^0) \ge \alpha^- B(x^-) + \alpha^+ B(x^+).
\end{equation}
\end{lem}

\begin{proof} For $p < \infty$, this is a direct calculation, since we simply need to check that the Hessian of $B$ is negative (semi-)definite.
\begin{gather*}
\frac{\partial r^{\pm}}{\partial x_1}(x) = \frac{(1-(p-1)r^{\pm})(1-pr^{\pm})}{x_1 r^{\pm} (p-1)}, \quad \frac{\partial r^{\pm}}{\partial x_2}(x) = \frac{-(1-(p-1)r^{\pm})(1-pr^{\pm})}{x_2 r^{\pm} p (p-1)}, \\
\frac{\partial B}{\partial r^{\pm}}(x) = -B(x) \left( \frac{q'(q'-1)r^{\pm}}{(1-pr^{\pm})(1-(p-1)r^{\pm})(1-\gamma r^{\pm})} \right), \\
\frac{\partial^2 B}{\partial x_1^2}(x) = \frac{-(1-(p-1)r^{\pm})^2\gamma q'(q'-1)B(x)}{(1-\gamma r^{\pm})(p-1)^2 r^{\pm} x_1^2}, \\
\frac{\partial^2 B}{\partial x_1 \partial x_2}(x) = \frac{(1-(p-1)r^{\pm})\gamma q'(q'-1)B(x)}{(1-\gamma r^{\pm})p(p-1)^2 r^{\pm} x_1 x_2}, \\
\frac{\partial^2 B}{\partial x_2^2}(x) = \frac{-\gamma q'(q'-1)B(x)}{(1-\gamma r^{\pm})p^2(p-1)^2 r^{\pm} x_2^2}.
\end{gather*}
The quadratic form given by the Hessian of $B$ is
\begin{align}\label{quadform}
\sum_{i,j=1}^2 &\frac{\partial^2 B}{\partial x_i \partial x_j}(x) \Delta_i \Delta_j = \notag \\
& \qquad \frac{-(1-(p-1)r^{\pm})^2 \gamma q'(q'-1) B(x)}{(1-\gamma r^{\pm})(p-1)^2 r^{\pm} x_1^2} \left( \Delta_1 - \frac{x_1}{(1-(p-1)r^{\pm})p x_2}\Delta_2\right)^2.
\end{align}
This is non-positive for $q > q^*$ because $(1- \gamma r^+) > 0$; for $\frac{p-1}{p} < q < q_*$, $\gamma < 0$ so this is non-positive as well. Thus, $B(x; p, q, \delta)$ is concave.

For $p = \infty$, a slightly different approach is needed. First of all, we may assume that $x^0_2 = x^+_2$ due to the symmetry between $x^-$ and $x^+$. Also, $B$ is linear in $x_1$, so
\begin{align*}
B(x^0) - \alpha^-B(x^-) - \alpha^+&B(x^+)\\
&= \alpha^-B(x^-_1, x^+_2) + \alpha^+B(x^+_1, x^+_2) - \alpha^-B(x^-) - \alpha^+B(x^+) \\
&= \alpha^-\left(B(x^-_1, x^+_2) - B(x^-_1, x^-_2)\right).
\end{align*}
This leads us to investigate
\[\frac{\partial B(x^-_1, x_2)}{\partial x_2} = \frac{x_2^{-q'}}{q-\delta} \left[q'\left(\delta \frac{x^-_1}{x_2} -1\right)\right],\]
which is non-negative, since we've assumed that $q > \delta$ and that $\frac{x^-_1}{x_2} \ge \frac{1}{\delta}$ for any $x_2$ between $x^-_2$ and $x^+_2$. Hence, $B(x; \infty, q, \delta)$ is concave.
\end{proof}

\section{How to find extremal weights}\label{weightsection}
We now want to show that $B(x) \le \B(x)$ on $\Omega_{\delta}$. To do so, given a point $x \in \Omega_{\delta}$, we will find a weight $w \in RH_p^{\delta}$ which corresponds to $x$ and which satisfies $B(x) = \avg{w^{1-q'}}$. We will (prematurely) call such a weight \emph{extremal}, because once we show that $B=\B$, these weights achieve the supremum which defines $\B$. The heuristics for finding such weights follows.

Since we know that $B$ is concave and that its Hessian has a kernel, we know that $B$ is linear along certain lines in $\Omega_{\delta}$. We will show later that these lines actually cover $\Omega_{\delta}$. With that, the heuristics above for why $B$ should be concave give us a pattern for how to find extremal weights. We start with $p < \infty$. Given an arbitrary point $x^0$ on the curve $\Gamma_{\delta}:=\{(x_1, x_2): x_2 = (\delta x_1)^p\}$, we find a maximal weight representing $x^0$ (this is far easier than doing so in general). Also, any weight represented by a point $x$ on the graph $\Gamma_1:=\{(x_1,x_2): x_2 = x_1^p\}$ is constant (and therefore maximal). So, given a point $\hat{x} \in \Omega_{\delta}$, we find the line along which $B$ is linear which passes through $\hat{x}$. This line will intersect the graphs $\Gamma_{\delta}$ and $\Gamma_1$ at points which we call $x^-$ and $x^+$, respectively. We find the constants $\alpha^{\pm}$ such that $\hat{x} = \alpha^- x^- + \alpha^+ x^+$. Then, $\hat{x}$ can be represented by the weight $w$ which is the concatenation of the maximal weight for $x^-$ on $I^-$ and the maximal (constant) weight for $x^+$ on $I^+$, re-scaling the intervals if necessary. Since $B$ is linear along this line, we know that $B(\hat{x}) = \alpha^- B(x^-) + \alpha^+ B(x^+) = \alpha^- \avg{w^{1-q'}}_{I^-} + \alpha^+ \avg{w^{1-q'}}_{I^+} = \avg{w^{1-q'}}_I,$ whence $w$ is maximal.

We start by finding an extremal weight for a point on the curve $\Gamma_{\delta}$. Given an arbitrary positive number $x_1^0$, let $x_2^0 = (\delta x_1^0)^p$. Let $I = [0,1]$. We seek a weight $w \in RH_p^{\delta}(I)$ such that $\avg{w}_I = x_1^0$, $\avg{w^p}_I = x_2^0$ and such that $\avg{w^{1-q'}}_I$ is as large as possible. Therefore, for any $a \in [0,1]$, we insist that
\[ \left(\frac{1}{a}\int_0^a w^p(t) \, dt \right) \left(\frac{1}{a} \int_0^a w(t) \, dt\right)^{-p} = \delta^p. \]
Let $v(a) := \int_0^a w^p(t) \, dt$. Then, $w(t) = v'(t)^{1/p}$. Therefore,
\begin{gather*}
\frac{v(a)}{a} = \delta^p \left(\frac{1}{a} \int_0^a v'(t)^{1/p} \, dt\right)^p, \text{ hence}\\
\frac{a}{\delta} \left(\frac{v(a)}{a}\right)^{1/p} = \int_0^a v'(t)^{1/p} \, dt.
\end{gather*}
We take the derivative with respect to $a$ and find
\begin{gather*}
 \delta^{-1} \left[ (1-\frac{1}{p}) t^{-1/p}v^{1/p} + \left(\frac{1}{p}\right)t^{1-1/p}v^{1/p - 1} v' \right] = (v')^{1/p}, \text{ so} \\
 p\delta = \left(p-1 + \frac{tv'}{v}\right) \left(\frac{v}{tv'}\right)^{1/p}.
 \end{gather*}
From the definition of $s^{\pm} = s^{\pm}_p(\delta)$, we know that
\[ p\delta = \left(p - 1 + \frac{1}{1-ps^{\pm}}\right) (1-ps^{\pm})^{1/p}, \]
so we must have $\frac{v}{tv'} = 1 - ps^{\pm}$. Consequently,
\begin{gather*}
\frac{dt}{t} \frac{1}{1-ps^{\pm}} = \frac{dv}{v} \text{, so} \\
v(t) = C t^{\frac{1}{1-ps^{\pm}}} \text{, whence} \\
w(t) = v'(t)^{1/p} = C t^{\frac{s^{\pm}}{1-ps^{\pm}}}.
\end{gather*}
We want $\avg{w}_I = x^0_1$, so we must set $C = x_1^0 \left(\frac{1-(p-1)s^{\pm}}{1-ps^{\pm}}\right)$. Putting it all together,
\begin{equation}\label{maxweight}
w(t) = x_1^0 \left(\frac{1-(p-1)s^{\pm}}{1-ps^{\pm}}\right)t^{\frac{s^{\pm}}{1-ps^{\pm}}}.
\end{equation}
It is straightforward to check that with this constant, $\avg{w^p}_I = x^0_2 = (\delta x_0^1)^p$. So, we have found our candidate for an extremal weight representing a point on the curve $\Gamma_{\delta}$.

Next, we must find the lines along which $B$ is linear. By \eqref{quadform}, we see that the vector field along which $B$ is linear is
\begin{equation}\label{vectorfield}
 p(1-(p-1)r^{\pm}) dx_1 - \frac{x_1}{px_2} dx_2 = 0.
\end{equation}
To find explicit formulae for the lines, we work with the definition of $r^{\pm}$. Recall
\begin{equation}\label{defnofr}
\frac{x_2}{x_1^p} = \delta^p \frac{(1-pr^{\pm})^{p-1}}{(1-(p-1)r^{\pm})^p}.
\end{equation}
Therefore,
\[\frac{dx_2}{x_2} - p \frac{dx_1}{x_1} = \frac{-rp(p-1)dr^{\pm}}{(1-pr^{\pm})(1-(p-1)r^{\pm})}.\]
Using \eqref{vectorfield}, we see that
\[\frac{dx_1}{x_1} = dr^{\pm} \left[ \frac{p}{1-pr^{\pm}} - \frac{p-1}{1-(p-1)r^{\pm}}\right] \quad \text{ and } \quad \frac{dx_2}{x_2} = \frac{pdr^{\pm}}{1-pr^{\pm}},\]
whence
\begin{equation}\label{coords}
x_1(r^{\pm}) = x_1(0) \frac{1-(p-1)r^{\pm}}{1-pr^{\pm}} \quad \text{ and } \quad x_2(r^{\pm}) = \frac{x_2(0)}{1-pr^{\pm}}.
\end{equation}
Taking $x_1(0) = b$ as a free parameter, \eqref{defnofr} gives us $x_2(0) = (\delta b)^p$. And, eliminating $r^{\pm}$ from \eqref{coords} yields
\begin{equation}
\delta^p p x_1 - b^{1-p}x_2 = \delta^p b (p-1),
\end{equation}
which is the equation of the line tangent to $\Gamma_{\delta}$ at the point $x = (b, (\delta b)^p)$. Notice that at $r^{\pm} = 0$, we have the point $(b, (\delta b)^p)$ on the line $\Gamma_{\delta}$ and at $r^{\pm}=s^{\pm}$, we have $x = \left(\frac{b(1-(p-1)s^{\pm})}{1-ps^{\pm}}, \frac{(\delta b)^p}{1-ps^{\pm}}\right)$, which satisfies $x_2 = x_1^p$ and hence lies on $\Gamma_1$. It is clear that by varying $b$, these segments cover $\Omega_{\delta}$.

We now check that $B$ is, in fact, linear along these segments. We use \eqref{B2} as an easier representation of $B$ for this calculation:
\begin{align*}
B(x) &= x_1^{-\gamma}x_2\left(\frac{1-ps^{\pm}}{1-pr^{\pm}}\right)^{\gamma}\left(\frac{1-(p-1)r^{\pm}}{1-(p-1)s^{\pm}}\right)^{\gamma}\left(\frac{1-\gamma r^{\pm}}{1-\gamma s^{\pm}}\right) \\
&= \left(\frac{b(1-(p-1)r^{\pm})}{1-pr^{\pm}}\right)^{-\gamma} x_2 \left(\frac{1-ps^{\pm}}{1-pr^{\pm}}\right)^{\gamma} \left(\frac{1-(p-1)r^{\pm}}{1-(p-1)s^{\pm}}\right)^{\gamma} \left(\frac{1-\gamma r^{\pm}}{1-\gamma s^{\pm}}\right) \\
&= \frac{1}{1-\gamma s^{\pm}}\left(\frac{1-ps^{\pm}}{b(1-(p-1)s^{\pm})}\right)^{\gamma}x_2(1-\gamma r^{\pm}).
\end{align*}
However, $x_2 = \frac{(\delta b)^p}{1-pr^{\pm}}$, so $r^{\pm} = \frac{1}{p} - \frac{(\delta b)^p}{p x_2}$ and therefore,
\[ x_2(1-\gamma r^{\pm}) = x_2(1-\frac{\gamma}{p}) + \frac{\gamma (\delta b^p)}{p}. \]
So,
\[ B(x) = \frac{1}{1-\gamma s^{\pm}}\left(\frac{1-ps^{\pm}}{b(1-(p-1)s^{\pm})}\right)^{\gamma}\left(x_2(1-\frac{\gamma}{p}) + \frac{\gamma (\delta b)^p}{p}\right),\] and hence is linear.

We now work with an arbitrary point $x^0 \in \Omega_{\delta}$. Given such a point, one of the segments on which $B$ is linear passes through $x^0$; finding it requires that we find the corresponding value of $b$. Further, given that $x^0$ corresponds to $I = [0,1]$, we want to calculate where to split $I$ so that $x^-$ is on $\Gamma_{\delta}$ and $x^+$ is on $\Gamma_1$. We determine $b$ first. $x^0$ determines a value $r^0 = u_p^{\pm}\left(\frac{x_2^0}{(\delta x_1^0)^p}\right)$, from which we get
\[x^0 = \left( \frac{b(1-(p-1)r^0)}{1-pr^0}, \frac{(\delta b)^p}{1-pr^0} \right) , \] and we see
\[b = \frac{x^0_1 (1- pr^0)}{1-(p-1)r^0}. \]
Then, if we split $I$ at $a$, we know that $x^0_2 - x^+_2 = a (x^-_2 - x^+_2)$. So, we calculate
\[ x^0_2 - x^+_2 = \delta^p b^p p \left(\frac{r^0 -s^{\pm}}{(1-pr^0)(1-ps^{\pm})}\right) \quad \text{ and } \quad x^-_2 - x^+_2 = \frac{-(\delta b)^p ps^{\pm}}{1-ps^{\pm}},\]
whence $a = \frac{s^{\pm}-r^0}{s^{\pm}(1-pr^0)}$.

On $I^+$, our weight should be constant, so
\[w(t)|_{I^+} \equiv x^+_1 = \frac{b(1-(p-1)s^{\pm})}{1-ps^{\pm}} = x^0_1 \frac{(1-pr^0)(1-(p-1)s^{\pm})}{(1-(p-1)r^0)(1-ps^{\pm})}.\]
On $I^-$, the weight should be maximal. So, we re-scale \eqref{maxweight} and get
\[w(t)|_{I^-} = b \left(\frac{1-(p-1)s^{\pm}}{1-ps^{\pm}}\right)\left(\frac{t}{a}\right)^{\frac{s^{\pm}}{1-ps^{\pm}}} = x^0_1 \frac{(1-pr^0)(1-(p-1)s^{\pm})}{(1-(p-1)r^0)(1-ps^{\pm})}\left(\frac{t}{a}\right)^{\frac{s^{\pm}}{1-ps^{\pm}}}.\] Therefore, our (potential) extremal weight is
\begin{equation}\label{extweight} w(t) = \begin{cases} x^0_1 \frac{(1-pr^0)(1-(p-1)s^{\pm})}{(1-(p-1)r^0)(1-ps^{\pm})}\left(\frac{t}{a}\right)^{\frac{s^{\pm}}{1-ps^{\pm}}}& 0\le t \le a \\
x^0_1 \frac{(1-pr^0)(1-(p-1)s^{\pm})}{(1-(p-1)r^0)(1-ps^{\pm})}& a \le t \le 1, \\
\end{cases}
\end{equation}
where $a = \frac{s^{\pm}-r^0}{s^{\pm}(1-pr^0)}$.

For the case of $p = \infty$, given the above work, finding the extremal weight is rather easy. If we look at the power of $t$ in \eqref{extweight} and use the asymptotics in \eqref{asymptotics}, we see the power of our extremal weight should be $\delta - 1$. Also, we want the weight to represent a given point $x^0 = (x^0_1, x^0_2)$; recall that the second coordinate when $p = \infty$ is just $x_2 = \esup_I w$. Thus, the constant part of the weight must be equal to $x^0_2$. This leaves us with the simple task of finding the appropriate splitting point $a$. We take
\begin{equation}\label{extweight2} w(t) = \begin{cases} x^0_2\left(\frac{t}{a}\right)^{\delta -1}& 0\le t \le a \\
x^0_2 & a \le t \le 1, \\
\end{cases}
\end{equation}
and look for an $a$ such that $\avg{w}_I = x^0_1$. But,
\begin{align*}
\avg{w}_I &= x^0_2 \int_0^a \left(\frac{t}{a}\right)^{\delta-1} \, dt + (1-a)x^0_2 \\
&= x^0_2 - a\left(x^0_2 - \frac{x^0_2}{\delta}\right).
\end{align*}
Therefore, $a = \frac{1 - x^0_1/x^0_2}{1 - 1/\delta}$.

\section{$B(x) \le \B(x)$}
\begin{lem}\label{extremalweightlemma} For every $\delta \ge 1$, $1 < p \le \infty$, $\frac{p-1}{p} < q < \infty$ and every $x \in \Omega_{\delta}$, $B(x; p, q, \delta) \le \B(x; p, q, \delta)$.
\end{lem}

\begin{proof}
As usual, we address $p < \infty$ first. Since $x_2 = x_1^p$ if and only if $w(t) \equiv x_1$ is constant, we know that for such points $r_p^{\pm} = s_p^{\pm}$ and hence that $\B(x_1, x_1^p; p, q, \delta) = B(x_1, x_1^p; p, q, \delta) = x_1^{1-q'}$ for all $x_1 > 0$, $1 < p < \infty$, and $0 < q < \infty$. So, we now consider $\delta > 1$ and points $x$ with $x_2 x_1^{-p} > 1$.

Fix an arbitrary point $x^0 \in \Omega_{\delta}$ with $x^0_2 (x^0_1)^{-p} > 1$. Fix $I = [0,1]$ and let $a \in (0,1]$. We define
\[w_{c,a,\nu}(t) = \begin{cases} c\left(\frac{t}{a}\right)^{\nu}& \text{ if } 0 \le t \le a \\
c& \text{ if } a \le t \le 1. \\ \end{cases} \]
Then,
\begin{equation}\label{avgofpower}
\avg{w^{\theta}_{c,a,\nu}}_I = \begin{cases} c^{\theta} \frac{1+(1-a)\theta\nu}{1+\theta\nu}& \text{ if } \theta\nu > -1 \\
\infty& \text{ if } \theta\nu \le -1. \\ \end{cases}
\end{equation}
The weight $w_{c,a,\nu}$ has $RH_p$ constant equal to $\frac{1+\nu}{(p\nu + 1)^{1/p}}$, as is demonstrated in the appendix (Lemma \ref{RHpnorm}).

As our earlier work suggests, to get an extremal function we use the values
\begin{equation}
\nu = \frac{s^{\pm}}{1-ps^{\pm}}, \quad a = \frac{s^{\pm}-r^0}{s^{\pm}(1-pr^0)}, \quad c = x^0_1 \frac{(1-pr^0)(1-(p-1)s^{\pm})}{(1-(p-1)r^0)(1-ps^{\pm})}.
\end{equation}
Since the calculations for both cases are very similar, we only address the case of $q > 1$. With $\delta > 1$, we know $s^+ > 0$; further, $s^+ < 1/p$. Therefore, $\frac{s^+-r^0}{s^+} < 1-pr^0$, whence $0 < a < 1$. And, $\nu > 0$. Also, for this value of $\nu$, the $RH_p$ constant for $w_{c,a,\nu}$ is equal to $\delta$, as can easily be checked. Using \eqref{avgofpower}, we check that the weight $w_{c,a,\nu}$ does indeed represent the point $x^0$. We see
\[
\avg{w_{c,a,\nu}}_I = c \frac{1 + (1-a)\nu}{1+\nu}, \quad \text{but}
\]
\[
1+\nu = \frac{1-(p-1)s^+}{1-ps^+}, \quad 1+(1-a)\nu = \frac{1-(p-1)r^0}{1-pr^0}, \text{ whence } \avg{w_{c,a,\nu}}_I = x^0_1.
\]
Further,
\[
\avg{w_{c,a,\nu}^p}_I = c^p \frac{1+(1-a)p\nu}{1+p\nu},
\]
\[
c^p = (x^0_1)^p \frac{\delta^p}{1-ps^+} \frac{x^0_2(1-pr^0)}{\delta^p (x^0_1)^p} = x^0_2 \frac{1-pr^0}{1-ps^+},
\]
\[
1+p\nu = \frac{1}{1-ps^+}, \quad 1+(1-a)p\nu = \frac{1}{1-pr^0}, \text{ thus } \avg{w_{c,a,\nu}^p}_I = x^0_2,
\]
and $w_{c,a,\nu}$ does represent $x^0$. Finally, we check that $w_{c,a,\nu}$ is maximal, assuming that $(1-q')\nu > -1$
\begin{gather*}
\avg{w_{c,a,\nu}^{1-q'}}_I = c^{1-q'} \frac{1+(1-a)(1-q')\nu}{1+(1-q')\nu}, \\
c^{1-q'} = (x^0_1)^{1-q'} \left(\frac{1-ps^+}{1-pr^0}\right)^{q'-1} \left(\frac{1-(p-1)r^0}{1-(p-1)s^+}\right)^{q'-1},
\end{gather*}
\[1+(1-q')\nu = \frac{1-\gamma s^+}{1-ps^+}, \quad 1+(1-a)(1-q')\nu = \frac{1-\gamma r^0}{1-pr^0}, \text{ hence } \avg{w_{c,a,\nu}^{1-q'}}_I = B(x^0).\]
Our assumption that $(1-q')\nu > -1$ yields the restrictions on $q$ in Theorems \ref{mainthm} and \ref{auxthm}:
\begin{gather*}
(1-q')\nu = \frac{(1-q')s^+}{1-ps^+}, \text{ so } \\
(1-q')\nu > -1 \Leftrightarrow q' < \frac{1-(p-1)s^+}{s^+} \Leftrightarrow q > \frac{1-(p-1)s^+}{1-ps^+} = q^*(\delta, p);
\end{gather*}
similarly, for $\frac{p-1}{p} < q < 1$, $(1-q')\nu > -1 \Leftrightarrow q < q_*$. Therefore, for $\frac{p-1}{p} < q < q_*$ and $q > q^*$, for all $p$ and any $\delta \ge 1$, by the definition of $\B(x)$, we know that $B(x; p, q, \delta) \le \B(x; p, q, \delta)$.

Also, for $q_* \le q \le q^*$, $\B(x)$ is infinite, since the average $\avg{w_{c,a,\nu}^{1-q'}}_I$ is infinite. Hence, for $q_* \le q \le q^*$, the inequality $B(x) \le \B(x)$ is trivially true.

Now, for $p = \infty$, we follow the same path. As before, along the line $x_1 = x_2$, the weights are all constant, so $\B(x; \infty, q, \delta) = B(x; \infty, q, \delta)$. We thus consider only $x_2 > x_1$. Again we fix a point $x^0 \in \Omega_{\delta}$ and consider the weight $w_{c, a, \nu}$. We use our earlier work to inform our choices of $c = x^0_2$, $a = \frac{1 - x^0_1/x^0_2}{1 - 1/\delta}$ and $\nu = \delta - 1$. We check that this weight represents $x^0$, first by checking
\[\avg{w_{c,a,\nu}}_I = c \frac{1 + (1-a)\nu}{1+\nu} = \frac{x^0_2}{\delta}(\delta - \delta(1-\frac{x^0_1}{x^0_2})) = x^0_1.\]
And, clearly, $\esup_I w_{c,a,\nu} = x^0_2$. This weight has $RH_{\infty}$ constant equal to $\delta$, which, as before, we prove in the appendix (see Lemma \ref{RHinftynorm}). Finally, we check that this weight is maximal, assuming $(1-q')\nu > -1$:
\begin{align*}
\avg{w_{c,a,\nu}^{1-q'}}_I &= c^{1-q'} \frac{1+(1-a)(1-q')\nu}{1+(1-q')\nu} = (x^0_2)^{1-q'} \frac{1+(1-a)(q' - \delta(q'-1) -1)}{q' - \delta(q'-1)} \\
&= (x^0_2)^{1-q'} \frac{q' - \delta(q'-1) + \delta(q'-1)(1- \frac{x^0_1}{x^0_2})}{q' - \delta(q'-1)} \\
&= (x^0_2)^{1-q'} \frac{q - \delta \frac{x^0_1}{x^0_2}}{q-\delta} = B(x^0).
\end{align*}
The restriction that $(1-q')\nu > -1$ corresponds to $q > \delta$, as in Theorem \ref{mainthm2}, and the fact that the average $\avg{w^{1-q'}_{c, a, \nu}}$ is infinite for $q \le \delta$ establishes the fact that $\B(x; \infty, q, \delta) \ge B(x; \infty, q, \delta)$ for all $q > 1$.
\end{proof}

\section{$B(x) \ge \B(x)$}\label{omegaepsilonsection}
We won't be able to prove this directly; instead, we will resort to an approximation procedure which will involve looking at domains $\Omega_{\epsilon}$ for $\epsilon > \delta$. We would like to use Lemma \ref{concavitylemma}, but there is a slight difficulty, in that the line joining two points in $\Omega_{\delta}$ mentioned there might leave $\Omega_{\delta}$. Thus, our first task is to show that for a given $\delta$, for every $\epsilon > \delta$, there is a way to split the interval in such a way that $[x^-, x^+]$ is contained inside $\Omega_{\epsilon}$.

\begin{lem}\label{splittinglemma}
\textbf{Case 1, $p < \infty$:} Fix $\delta > 1$. Then for an arbitrary $\epsilon > \delta$ and an arbitrary weight $w \in RH_p^{\delta}(J)$, there is a splitting $J = J^- \cup J^+$, $|J^{\pm}| = \alpha^{\pm}|J|$ such that the entire interval with the endpoints $x^{\pm} = \left(\avg{w}_{J^{\pm}}, \avg{w^p}_{J^{\pm}}\right)$ is in $\Omega_{\epsilon}$. Moreover, the splitting parameters $\alpha^{\pm}$ can be chosen bounded away from 0 and 1 uniformly with respect to $w$ and, therefore, with respect to $J$ as well.

\textbf{Case 2, $p = \infty$:} Fix $\delta > 1$. Then for an arbitrary $\epsilon > \delta$ and an arbitrary weight $w \in RH_{\infty}^{\delta}(J)$, there is a splitting $J = J^- \cup J^+$, $|J^{\pm}| = \alpha^{\pm}|J|$ such that $\frac{x^0_2}{x^{\pm}_1} \le \epsilon$, where $x^{\pm} = \left(\avg{w}{J^{\pm}}, \esup_{J^{\pm}} w\right)$ and $x^0 = \left(\avg{w}_J, \esup_J w\right)$. Moreover, the splitting parameters $\alpha^{\pm}$ can be chosen bounded away from 0 and 1 uniformly with respect to $w$ and, therefore, with respect to $J$ as well.
\end{lem}
\begin{proof}
We start with $p < \infty$. Picking a weight $w \in RH_p^{\delta}(J)$ fixes an interval $J$ and a point $x^0$ in $\Omega_{\delta}$. Starting from this point, we first try $\alpha^{\pm} = \frac{1}{2}$. If the entire interval $[x^-, x^+] \subset \Omega_{\epsilon}$ for these parameters, then we fix this splitting and stop. Assuming that some point of $[x^-, x^+]$ is outside of $\Omega_{\epsilon}$ for $\alpha^{\pm}= \frac{1}{2}$, we then consider the possible points of escape from $\Omega_{\epsilon}$. First of all, since $w \in RH_p^{\delta}(J)$, $x^0$, $x^+$ and $x^-$ are all contained within $\Omega_{\delta}$. As they are co-linear, and as the boundary graph $\Gamma_{\epsilon}:=\{(x_1,x_2): x_2 = \epsilon^p x_1^p\}$ is convex, we know that the portion of $[x^-, x^+]$ which lies outside of $\Omega_{\epsilon}$ must be either between $x^-$ and $x^0$ or between $x^0$ and $x^+$, but not both. Let $\xi = x^-$ in the first case and $\xi = x^+$ in the second case; that is, the only part of $[x^-, x^+]$ which lies outside of $\Omega_{\epsilon}$ is contained within the segment $[x^0, \xi]$. We now proceed to change $\alpha^+$ to bring $[x^-, x^+]$ entirely within $\Omega_{\epsilon}$. If $\xi = x^-$, we decrease $\alpha^+$; if $\xi = x^+$, we increase $\alpha^+$. Let $\rho(\alpha^+)$ be the maximum value of $\frac{x_2^{1/p}}{x_1}$ along the segment $[x^0, \xi]$. We already know that $\rho(\frac{1}{2}) > \epsilon$, and, since $\epsilon > \delta$, for $\xi$ sufficiently close to $x^0$, $\rho(\alpha^+) < \epsilon$. Further, $\rho(\alpha^+)$ is continuous. Therefore, when changing $\alpha^+$ from $\frac{1}{2}$, there is a value of $\alpha^+$ such that $\rho(\alpha^+) = \epsilon$ for the first time; call this ``stopping time'' value $\omega^+$ (with its corresponding $\omega^-$ such that $\omega^+ + \omega^- = 1$). We now check that $\omega^+$ is bounded away from 0 and 1 uniformly with respect to $w$ and $I$.

If $\xi = x^+$, then $\omega^+ > \frac{1}{2}$ and $\xi_1 - x^0_1 = x^+_1 - x^0_1 = \omega^- (x_1^+ - x_1^-)$. On the other hand, if $\xi = x^-$, then $\omega^- > \frac{1}{2}$ and $\xi_1 - x^0_1 = x^-_1 - x^0_1 = \omega^+ (x^-_1 - x^+_1)$. Thus, $|\xi_1 - x^0_1| = \min\{\omega^{\pm}\}|x^-_1 - x^+_1|$.

At $\omega^+$, the line passing through $x^{\pm}$ and $x^0$ is tangent to $\Gamma_{\epsilon}$ and touches it at a point we will call $\tau$. The equation for the line tangent to $x_2 = c x_1^p$ for any constant $c$ at the point $\tau$ is \[p \tau_2 x_1 - \tau_1 x_2 = (p-1) \tau_1 \tau_2. \]
The equation for the points of intersection of this line with the graphs $\Gamma_{\delta}$ and $\Gamma_1$, reduces to
\[ t (\frac{x_1}{\tau_1})^p - p \frac{x_1}{\tau_1} = 1-p, \]
for some fixed $t$ which has the solutions $x_1 = \tau_1 \left( 1 + \frac{u_p^{\pm}(t)}{1-pu_p^{\pm}(t)} \right)$. For $\Gamma_{\delta}$, $t = \left(\frac{\delta}{\epsilon}\right)^p$ and for $\Gamma_1$, $t = \frac{1}{\epsilon^p}$. Therefore, the line tangent to $\Gamma_{\epsilon}$ at $\tau$ intersects $\Gamma_{\delta}$ at the points
\[x(\delta)^{\pm} = \left(\tau_1 \left[1 + \frac{u_p^{\pm}((\frac{\delta}{\epsilon})^p)}{1-pu_p^{\pm}((\frac{\delta}{\epsilon})^p)}\right], \tau_2 \left[\frac{1}{1-pu_p^{\pm}((\frac{\delta}{\epsilon})^p)}\right]\right),\]
and it intersects $\Gamma_1$ at
\[x(1)^{\pm} = \left(\tau_1 \left[1 + \frac{u_p^{\pm}(\frac{1}{\epsilon^p})}{1-pu_p^{\pm}(\frac{1}{\epsilon^p})}\right], \tau_2 \left[\frac{1}{1-pu_p^{\pm}(\frac{1}{\epsilon^p})}\right]\right).\]
This gives us the following string of inclusions, $[x(\delta)^-, x(\delta)^+] \subset [x^0, \xi] \subset [x^-, x^+] \subset [x(1)^-, x(1)^+]$. Therefore,
\[|x(\delta)^+_1 - x(\delta)^-_1| \le |x^0_1 - \xi_1| = \min\{\omega^{\pm}\}|x^+_1 - x^-_1| \le \min\{\omega^{\pm}\}|x(1)^+_1 - x(1)^-_1|, \]
and
\begin{align*}
\min\{\omega^{\pm}\} &\ge \frac{|x(\delta)^+_1 - x(\delta)^-_1|}{|x(1)^+_1 - x(1)^-_1|} \\
&= \frac{|(u_p^+((\frac{\delta}{\epsilon})^p) - u_p^-((\frac{\delta}{\epsilon})^p))(1-pu_p^+(\frac{1}{\epsilon^p}))(1-pu_p^-(\frac{1}{\epsilon^p}))|} {|(u_p^+(\frac{1}{\epsilon^p}) - u_p^-(\frac{1}{\epsilon^p})) (1-pu_p^+((\frac{\delta}{\epsilon})^p)) (1-pu_p^-((\frac{\delta}{\epsilon})^p))|},
\end{align*}
which is bounded away from zero and depends only on $p$, $\delta$ and $\epsilon$, and neither $w$ nor $I$.

Now, we turn to the case of $p = \infty$. Since $x^0_2 = \text{max}\{x^{\pm}_2\}$ and both points $x^{\pm}$ are in $\Omega_{\delta}$, at least one of the inequalities $\frac{x^0_2}{x^{\pm}_1} \le \epsilon$ is always true. First we take $\alpha^- = \alpha^+ = 1/2;$ if the required inequalities are both true, we fix this splitting. Otherwise, we start to change the splitting; namely, we increase $\alpha^+$ if the point $(x^+_1, x^0_2)$ is outside $\Omega_{\epsilon}$ and reduce it in the opposite case. By symmetry, it suffices to examine one of the possible situations, say, the case where $\frac{x^0_2}{x^+_1} > \epsilon$ for $\alpha^+= 1/2$. The points $x^{\pm}$ do not, in general, depend continuously on $\alpha^+$, but the first coordinates $x^{\pm}_1$ do. Since $x^0 \in \Omega_{\delta}$ and $\epsilon > \delta$, for all $\alpha^+$ sufficiently close to 1, $\frac{x^0_2}{x^+_1} < \epsilon$, and, by our assumption, $\frac{x^0_2}{x^+_1} > \epsilon$ for $\alpha^+ = 1/2$. So, when we increase $\alpha^+$ from $1/2$ (i.e., enlarge $I^+$), there is a value of $\alpha^+$ such that $\frac{x^0_2}{x^+_1} = \epsilon$ for the first time. As before, we call this ``stopping time'' $\omega^+$, with $\omega^- := 1- \omega^+$. We want to check that $\omega^{\pm}$ are bounded away from 0 and 1, using the geometry of the situation. At the stopping time, we know $\epsilon x^+_1 = x^0_2$. Also, since $x^0_2 = x^-_2$, we know that $x^0$ and $x^-$ both lie on the horizontal line through the point $(x^+_1, x^0_2)$. Moreover, since $x^- \in \Omega_{\delta}$, we know that $\frac{x^0_2}{\delta} \le x^-_1 \le x^0_2$.
Therefore, by examining the first coordinates of these points, we see
\[ 0 < \left(\frac{1}{\delta} - \frac{1}{\epsilon}\right)x^0_2 = \frac{x^0_2}{\delta} - x^+_1 \le x^0_1 - x^+_1 = \omega^- (x^-_1 - x^+_1) \le \omega^- \left(1 -\frac{1}{\epsilon}\right) x^0_2, \]
from which we see that $\omega^- \ge \frac{\epsilon/\delta - 1}{\epsilon - 1} > 0.$ Since $\omega^+ > 1/2$, this proves that $\omega^{\pm}$ are both bounded away from 0 and 1. Further, these bounds depend on $\delta$ and $\epsilon$, and on neither $w$ nor $J$.
\end{proof}

We now prove the equality of $\B$ and $B$ by establishing the following inequality.
\begin{lem}\label{omegaepsilonlemma}
For every $\delta \ge 1$, every $x \in \Omega_{\delta}$, every $p \in (1, \infty]$, every $q \in (\frac{p-1}{p}, \infty)$ and every $\epsilon > \delta$,
\[\B(x; p, q, \delta) \le B(x; p, q, \epsilon) \]
\end{lem}
With this established, we can pass to the limit as $\epsilon \to \delta$. For $p < \infty$ and $q > 1$, this is so because $s^+$ is a continuous, increasing function of $\delta$. Thus, if $s^+(\delta) \ge \frac{1}{\gamma}$, then $s^+(\epsilon)$ is also, in which case $B(x; p, q, \epsilon)$ is infinite and the inequality is trivially true. On the other hand, if $s^+(\delta) < \frac{1}{\gamma}$, then there is a $\kappa > 0$ such that $s^+(\delta+\kappa) = \frac{1}{\gamma}$, and as long as $|\epsilon - \delta| < \kappa$, we know that $B(x; p,q, \epsilon)$ is finite and depends continuously upon $\epsilon$. A similar argument applies to $\frac{p-1}{p} < q < 1$. For $p = \infty$, the continuity in $\delta$ is clear; if $\delta \ge q$, then $B(x; \infty, q, \epsilon)$ is infinite. If $\delta < q$, then as long as $|\epsilon - \delta| < (q - \delta)$, $B(x; \infty,q,\epsilon)$ is finite and depends continuously on $\epsilon$. The following proof is due to Vasyunin \cite{vasyunin}.

\begin{proof}[Proof of Lemma \ref{omegaepsilonlemma}]
The statement of the lemma means that for an arbitrary weight $w \in RH_p^{\delta}(J)$ and for any $\epsilon > \delta$, we have
\begin{equation}\label{epsiloncondition}
\avg{w^{1-q'}}_J \le B(x; p, q, \epsilon).
\end{equation}
Here $x = (x_1, x_2) = (\avg{w}_J, \avg{w^p}_J)$. It suffices to prove \eqref{epsiloncondition} for step functions $w$ because an arbitrary weight $w$ can be approximated by a sequence of step functions $w_n$ so that all averages converge: $\avg{w_n^{1-q'}}_J \to \avg{w^{1-q'}}_J$, and $(\avg{w_n}_J, \avg{w_n^p}_J) \to (\avg{w}_J, \avg{w^p}_J)$, and we can pass to the limit in \eqref{epsiloncondition} because $B$ is continuous in $x$ if it is finite, and if it is infinite then there is nothing to prove.
So, we fix an interval $J$, a step function $w \in RH_p^{\delta}(J)$ and a number $\epsilon > \delta$. Starting with the interval $J^{0,0}$, we construct a chain of intervals $J^{n,m}$ in accordance with the rule in Lemma \ref{splittinglemma}. For the intervals in the $n$th generation, we use the notation $J^{n,m}$, where $J^{n,2k} = (J^{n-1,k})^-$ and $J^{n,2k+1}=(J^{n-1,k})^+$. Consequently, the second index runs from $0$ to $2^n -1$. The corresponding mean values will be labeled by the same pair of indices and $\alpha^{n,m} = |J^{n,m}|/|J|$. By Lemma \ref{concavitylemma}, we can write the following chain of inequalities:
\begin{align}
B(x^{0,0}; p, q, \epsilon) &\ge \alpha^{1,0} B(x^{1,0}; p, q, \epsilon) + \alpha^{1,1}B(x^{1,1}; p, q, \epsilon) \notag \\
&\ge \sum_{m=0}^{2^n-1}\alpha^{n,m} B(x^{n,m}; p, q, \epsilon).
\end{align}
The latter sum tends to $\avg{w^{1-q'}}_J$ as $n \to \infty$. Indeed, for a fixed step function $w \in RH_p^{\delta}(I)$, the set $\{ x = (\avg{w}_J,\avg{w^p}_J): J \subset I\}$ is a compact subset of $\Omega_{\delta}$, and, therefore, the continuous function $B$ is bounded on this subset, say $B \le M$, i.e., $B(x^{n,m}; p, q, \epsilon) \le M$ (excluding the case of $B = \infty$, when there's nothing to prove). So if $N$ is the number of discontinuity points for the step function $w$, then the number of intervals $J^{n,m}$ where $w$ is not a constant is at most $N$. On all other intervals $w$ is a constant, in which case $B(x^{n,m}; p, q, \epsilon) = (x_1^{n,m})^{1-q'} = w^{1-q'}$, i.e., the corresponding summand is $\avg{w^{1-q'}}_{J^{n,m}}$ and the entire sum differs from $\avg{w^{1-q'}}_J$ by at most $NM \max_m \alpha^{m,n}$. This latter quantity tends to zero because by Lemma \ref{splittinglemma}, the values $\alpha^{\pm}$ are bounded away from 0 and 1, and the maximum length of the $n$th generation intervals tends to 0 as $n \to \infty$.
For $p=\infty$, the proof is nearly identical; simply use $x_2 = \esup_J w$ instead of $x_2 = \avg{w^p}_J$.
\end{proof}

\section{$n$-dimensional results}\label{ndimsection}
We turn to the proof of Theorem \ref{ndimthm}. Lemma \ref{ndimlemma}, below, and the maximization argument from section \ref{thmproofsection} yield an upper bound on the Bellman function. This gives us an upper bound, $C(\delta, p, q, n)$, on the $A_q$ constant of the weights in question. We finish the proof by showing that for a fixed $n>1$, $p>1$ and $q>1$, the limit of $C(\delta, p, q, n)$ as $\delta \to 1$ is 1.

We only consider $n$-cubes (equal side lengths) with sides parallel to the coordinate axes. Each time we divide a cube $Q$, we cut it into $2^n$ sub-cubes $\{Q_i\}$, each of size $|Q_i| = \frac{|Q|}{2^n}$. Given a cube $Q \subset \R^n$, a weight $w \in RH_p^{\delta}(Q)$ if \[\frac{\avg{w^p}_{K}^{1/p}}{\avg{w}_{K}} \le \delta\] for all sub-cubes $K \subset Q$. Restricting ourselves to cubes doesn't allow us the flexibility we had before in choosing our splitting, and it is here that we lose the sharpness of the constant.

Given a cube $Q$ and its $2^n$ sub-cubes $\{Q_i\}$, we first find bounds on the ratio $\frac{\avg{w}_{Q_i}}{\avg{w}_{Q_j}}$ for $\delta$ sufficiently small (Lemma \ref{ratiolemma}). We can then use this bound to find a domain $\Omega_{\epsilon}$ which contains the convex hull of the points $\{(\avg{w}_{Q_i}, \avg{w^p}_{Q_i})\}$ which arise from the splitting of $Q$. Finally, we use the concavity of our function to produce an upper bound for the Bellman function.

\begin{lem}\label{ratiolemma}
Fix $n > 1$, a cube $Q \subset \R^n$ and a weight $w \in RH_p^{\delta}(Q)$. Divide $Q$ into $2^n$ equal sub-cubes $Q_i$ in the manner described above. If $1 \le \delta < \left(\frac{2^n}{2^n - 1}\right)^{1/p'}$, then
\[\frac{1}{y} \le \frac{\avg{w}_{Q_i}}{\avg{w}_{Q_j}} \le y,\]
where $y$ is given as the solution to
\begin{equation}\label{defnofy}
\left( 2+ 2^n(\delta^{-p'} - 1)\right)^{p-1} = \frac{(1+y)^p}{1+y^p}.
\end{equation}
\end{lem}

\begin{proof} First, an observation:
\begin{equation}\label{avg_inequal}
\frac{\sum_i \avg{w}_{Q_i}^p}{(\sum_i \avg{w}_{Q_i})^p} \le \frac{\delta^p}{2^{n(p-1)}}.
\end{equation}
This can be seen from the following, where we apply H\"older's inequality and then the assumed reverse-H\"older's inequality.
\begin{align*}
\frac{1}{2^n} \sum_{i=1}^{2^n} \avg{w}_{Q_i}^p &\le \frac{1}{2^n} \sum_{i=1}^{2^n} \avg{w^p}_{Q_i} = \avg{w^p}_Q \notag \\
&\le \delta^p \avg{w}_Q^p = \delta^p \left(\frac{1}{2^n} \sum_{i=1}^{2^n}\avg{w}_{Q_i} \right)^p.
\end{align*}
The idea is that if $\delta$ is sufficiently small, \eqref{avg_inequal} restricts the size of the quantity $\frac{\avg{w}_{Q_i}}{\avg{w}_{Q_j}}$.

Temporarily let $x_i := \avg{w}_{Q_i},$ $i=1 \ldots 2^n$. Then arrange the $x_i$ in non-decreasing order and re-label so that $x_1 \le x_2 \le \cdots \le x_{2^n}$. Now, set $y_i := \frac{x_i}{x_1}$. Then $1 = y_1 \le y_2 \le \cdots \le y_{2^n}$. We are then after the solution to the following optimization problem: given the set
\[S:= \left\{ (y_1, y_2, \ldots, y_{2^n}) : 1 = y_1 \le y_2 \le \cdots \le y_{2^n} \textrm{ and } \frac{\sum_i y_i^p}{(\sum_i y_i)^p} \le \frac{\delta^p}{2^{n(p-1)}}\right\},\] what is
\[\sup \, \{ y_{2^n} : \exists (y_1, y_2, \ldots, y_{2^n}) \in S \}?\]
If this supremum is finite, then we have a bound on the ratio between the averages of our weight on different sub-cubes $Q_i$.

Define the function $k(y_1, y_2, \ldots, y_{2^n}) := \frac{\sum_i y_i^p}{(\sum_i y_i)^p}$. A little calculus (and induction) allows one to see that, for any $n > 1$, given a value $y:= y_{2^n}$, the minimum value of $k(1, x_2, \ldots, x_{2^n-1}, y)$ on the region defined by $1 \le x_2 \le \cdots \le x_{2^n-1} \le y$ is at the point $(1, a, a, \ldots, a, y)$, where $a := \left(\frac{1+y^p}{1+y}\right)^{1/(p-1)}$. Thus, if we choose our value $y$ so that $k(1, a, \ldots, a, y) = \frac{\delta^p}{2^{n(p-1)}}$, we have found our supremum. This value of $y$ solves \eqref{defnofy}, which is only possible (for $1 \le y < \infty$) if $1 \le \delta < \left(\frac{2^n}{2^n - 1}\right)^{1/p'}$.
\end{proof}

For any weight $w \in RH_p^{\delta}(Q)$, the set of points
\[P:= \{(\avg{w}_Q, \avg{w^p}_Q), (\avg{w}_{Q_1}, \avg{w^p}_{Q_1}), \ldots, (\avg{w}_{Q_{2^n}}, \avg{w^p}_{Q_{2^n}})\}\] lies in the domain $\Omega_{\delta} := \{(x_1,x_2) : x_1^p \le x_2 \le (\delta x_1)^p \}$. Our goal is to find an $\epsilon$ such that the convex hull of $P$ lies within $\Omega_{\epsilon}$. Since the curve $\Gamma_{\delta} := \{(x_1, (\delta x_1)^p)\}$ is convex, if part of the convex hull of $P$ lies outside of $\Omega_{\delta}$, only one part of one edge of the hull lies the furthest outside of $\Omega_{\delta}$. Thus, we can simply focus on pairs of points in $P$.

Since $\frac{1}{y} \avg{w}_{Q_i} \le \avg{w}_{Q_j} \le y \avg{w}_{Q_i}$, we know $\frac{1}{y} \avg{w}_Q \le \avg{w}_{Q_j} \le y \avg{w}_Q$. If we label $\avg{w}_Q = x_1^0$, as before, then we see that the worst that could happen is if $P_1:=(\frac{x_1^0}{y}, (\delta \frac{x^0_1}{y})^p)$ and $P_2:=(yx^0_1, (\delta y x_1^0)^p)$ are both in $P$. Thus, the smallest $\epsilon$ which guarantees that $\Omega_{\epsilon}$ will contain the convex hull of $P$ is such that the line between $P_1$ and $P_2$ is tangent to $\Gamma_{\epsilon}$.
Solving for $\epsilon$ yields
\begin{equation}\label{defnofepsilon}
\epsilon = \delta \left[ \frac{f_p(y)}{p} \cdot \left(\frac{f_p(y)-1}{p-1}\right)^{(1-p)/p} \right],
\end{equation}
where $f_p(y) := \frac{y^2 - y^{2-2p}}{y^2-1}$. As long as $y \ge 1$, \eqref{defnofepsilon} gives that $\epsilon \ge \delta$ and is bounded.

Consequently, if $\delta < \left(\frac{2^n}{2^n-1}\right)^{1/p'}$, there is a value $y(\delta,n,p)$ which solves \eqref{defnofy}. Then, the $\epsilon$ which satisfies our conditions is given by \eqref{defnofepsilon}, using $y(\delta, n, p)$, and this $\epsilon$ depends only upon $\delta, n$ and $p$.

Now, we proceed as before. The function $B(x; p, q, \delta)$ used in the one-dimensional case is still our ``best guess" at the true Bellman function, as the scaling argument and calculations of section \ref{Bguesssection} are identical in the $n$-dimensional case. What changes is that we can only prove a restricted version of Lemma \ref{omegaepsilonlemma}.

\begin{lem}\label{ndimlemma}
For every $n \ge 1$, every $p, q \in (1, \infty)$, every $\delta \in [1,\left(\frac{2^n}{2^n-1}\right)^{1/p'})$, every $x \in \Omega_{\delta}$, if $\epsilon$ solves \eqref{defnofepsilon}, with the value $y(\delta, n, p)$ given by \eqref{defnofy}, then
\[\B(x; p, q, \delta) \le B(x; p, q, \epsilon) \]
\end{lem}

The proof of this lemma, given the above work, is actually easier than the proof of Lemma \ref{omegaepsilonlemma}. However, the argument is so similar that we won't repeat it here. The main change is that in the $n$-dimensional case, the splitting is determined and uniform.

Now, from Lemma \ref{ndimlemma}, we can get an upper bound on the $A_q(Q)$ constant of a weight $w \in RH_p^{\delta}(Q)$.
\begin{equation}
\sup_{K \subset Q}\,\avg{w}_K\avg{w^{1-q'}}^{q-1}_K \le C(\delta, p, q, n),
\end{equation}
with \[C(\delta, p, q, n)= \begin{cases}
+\infty & 1< q \le q^{*}(p, \epsilon) \\
\frac{1}{q^{*}(p, \epsilon)}\left(\frac{q-1}{q-q^{*}(p, \epsilon)}\right)^{q-1} & q > q^{*}(p, \epsilon) \\
\end{cases},\]
and where $q^{*}(p, \epsilon)$, is as defined in \eqref{qstardef}, but with $\delta$ replaced by the $\epsilon$ which solves \eqref{defnofepsilon}, given the $y$ which is a solution to \eqref{defnofy}.

What is important for the proof of Theorem \ref{ndimthm} is the limit of this bound as $\delta$ approaches 1 for a fixed $p>1$, $q>1$, $n>1$. It is not difficult to see, from \eqref{defnofy}, that for a fixed $n>1$ and $p>1$, $\lim_{\delta \to 1} y(\delta,n,p) = 1$. Also, for a fixed $p>1$, the function $f_p(y) = \frac{y^2 - y^{2-2p}}{y^2 - 1}$ satisfies $\lim_{y \to 1} f_p(y) = p$. Consequently, the limit of the $\epsilon$ given by \eqref{defnofepsilon} as $\delta \to 1$ is 1. By our earlier work on $q^*$, we know that for a fixed $p>1$, $\lim_{\epsilon \to 1} q^*(p, \epsilon) = 1$, whence $\lim_{\delta \to 1} C(\delta, p, q, n) = 1$. Therefore, given any $n>1$, $p>1$, $q>1$, and $\eta >1$, by taking $\delta$ close enough to 1, we can ensure that every weight $w \in RH_p^{\delta}(Q)$ satisfies
\[\sup_{K \subset Q}\,\avg{w}_K\avg{w^{1-q'}}^{q-1}_K \le \eta,\] whence $RH_p^{\delta}(Q) \subset A_q^{\eta}(Q)$. This proves Theorem \ref{ndimthm}.

\begin{proof}[Proof of Theorem \ref{strongRHpthm}]
Our earlier work is nearly sufficient; as the Bellman function is dimension-blind, only the splitting of the rectangles and the extremal weights need to be addressed. At the start of the proof of Lemma \ref{splittinglemma}, given the bounded rectangle $I \subset \R^n$ and a weight $w \in \str RH_p^{\delta}(I)$, re-scale $I$ so that that the longest side(s) of $I$ has length 1. Also, translate $I$ so that (one of) the longest side(s) is the interval $(0,1)$ in the direction which we will distinguish with the label $x_1$ (as before, translating and re-scaling $I$ doesn't affect the Bellman function). Then, split $I$ by cutting this $x_1$ side a distance $0 < \alpha^- < 1$ from 0, producing two sub-rectangles $I^{\pm}$. As before, $|I^{\pm}|= \alpha^{\pm} |I|$ (where $\alpha^- + \alpha^+ = 1$), and we get two weights $w^{\pm}$ defined on $I^{\pm}$. The remainder of the splitting is done by further sub-dividing $I$ along the $x_1$ axis. Given this convention, the proof of Lemma \ref{splittinglemma} is exactly the same.

Moreover, the extremal weights we found earlier are sufficient here; for a point $x \in \Omega_{\delta}$, an extremal weight representing $x$ on the cube $I := (0,1)^n$ is simply
\[w(x_1, x_2, \ldots, x_n) = w_{c,a,\nu}(x_1) = \begin{cases} c\left(\frac{x_1}{a}\right)^{\nu}& \text{ if } 0 \le x_1 \le a \\
c& \text{ if } a \le x_1 \le 1. \\ \end{cases}, \]
with $c, a, \nu$ as before. This is not difficult to check, and it completes the proof of Theorem \ref{strongRHpthm}.
\end{proof}

\section{Appendix: Finding the $RH_p$ and $RH_{\infty}$ constants}
\begin{lem}\label{RHpnorm}
The $RH_p$ constant for the weight
\[w_{c,a,\nu}(t) = \begin{cases} c\left(\frac{t}{a}\right)^{\nu}& \text{ if } 0 \le t \le a \\
c& \text{ if } a \le t \le 1, \\ \end{cases} \]
with $0 < a \le 1$, $c \ne 0$ and $\nu > -\frac{1}{p}$ is $\frac{1+\nu}{(p\nu + 1)^{1/p}}$.
\end{lem}
\begin{proof}
We want to find the supremum of the expression
\begin{equation}\label{Bpratio}
\frac{\avg{w^p}_J^{1/p}}{\avg{w}_J}
\end{equation}
over all intervals $J \subset I = [0,1]$. We first notice that the value of $c$ is immaterial, as the ratio \eqref{Bpratio} is invariant if we multiply $w$ by a constant. Therefore, we simplify the calculations and set $c=1$. It suffices to restrict our attention to intervals $J = [\alpha, \beta]$, with $0 \le \alpha < a \le \beta \le 1$. We will justify this restriction later. We now simply calculate
\begin{gather*}
\avg{w_{a,\nu}}_{[\alpha,\beta]} = \frac{\beta(\nu+1) - a\nu - a^{-\nu}\alpha^{\nu+1}}{(\beta - \alpha)(\nu + 1)} \\
\avg{w_{a, \nu}^p}_{[\alpha, \beta]} = \frac{\beta(p\nu + 1) - ap\nu - a^{-p\nu}\alpha^{p\nu +1}}{(\beta - \alpha)(p\nu + 1)},
\end{gather*}
so
\begin{equation}\label{star}
\frac{\avg{w_{a,\nu}^p}^{1/p}_{[\alpha,\beta]}}{\avg{w_{a,\nu}}_{[\alpha,\beta]}} =  \frac{(\nu + 1)(\beta - \alpha)^{1-1/p}(\beta(p\nu + 1) - ap\nu - a^{-p\nu}\alpha^{p\nu +1})^{1/p}}{(p\nu + 1)^{1/p}(\beta(\nu+1) - a\nu - a^{-\nu}\alpha^{\nu+1})}.
\end{equation}
If we define
\[ \lambda := \frac{a^{-\nu}\alpha^{\nu + 1}}{\beta(\nu + 1) - a \nu} \quad \text{ and } \quad \mu := \frac{a^{-p\nu}\alpha^{p\nu +1}}{\beta(p\nu + 1) - ap\nu},\]
we see that $0 \le \lambda < 1$ and $0 \le \mu < 1$, by our restrictions on $\alpha$, $\beta$ and $\nu$. Also, define
\[ \theta := \frac{1+\nu}{(p\nu + 1)^{1/p}}. \]
Using these substitutions, and pulling out $\beta$, the right-hand side of \eqref{star} becomes
\[\theta \frac{(1-\mu)^{1/p}}{1-\lambda}\left(1-\frac{\alpha}{\beta}\right)^{1-1/p} \frac{(1 + p\nu(1-\frac{a}{\beta}))^{1/p}}{1+\nu(1-\frac{a}{\beta})}.\]
We simplify further, using $\tau = \nu (1- \frac{a}{\beta})$, and $K = \frac{(1+p\tau)^{1/p}}{1+\tau}$. Note that for $\nu > -\frac{1}{p}$, we have $0 < K \le 1$. Then we arrive at
\begin{equation}
\frac{\avg{w_{a,\nu}^p}^{1/p}_{[\alpha,\beta]}}{\avg{w_{a,\nu}}_{[\alpha,\beta]}} = \theta \frac{(1-\mu)^{1/p}}{1-\lambda} \left(1-\frac{\alpha}{\beta}\right)^{1-1/p} K.
\end{equation}
One further reduction is possible, as
\begin{equation}\label{alphabetarelation} \lambda^p \mu^{-1} = \beta^{1-p}\alpha^{p-1} K^p, \text{ whence } \frac{\alpha}{\beta} = \left(\frac{\lambda}{K}\right)^{p'}\mu^{\frac{-1}{p-1}}
\end{equation}
So, we arrive at the expression we want to maximize,
\[\theta \frac{(1-\mu)^{1/p}}{1-\lambda}\left(1-\left(\frac{\lambda}{K}\right)^{p'}\mu^{-\frac{1}{p-1}}\right)^{1-1/p} K.\]
Let
\[\phi(\lambda, \mu) = \frac{(1-\mu)^{1/p}}{1-\lambda}\left(1-\left(\frac{\lambda}{K}\right)^{p'} \mu^{-\frac{1}{p-1}}\right)^{1-1/p},\]
and the rest is straightforward calculus.
\[\frac{\partial \phi}{\partial \mu} = \frac{(1 - \left(\frac{\lambda}{K}\right)^{p'} \mu^{-\frac{1}{p-1}})^{-1/p}(1-\mu)^{-1/p'}} {p(1-\lambda)}\left[\left(\frac{\lambda}{K\mu}\right)^{p'}-1\right],\]
which is zero at $\mu = \lambda/K$; further, this critical value is the location of a maximum. Consequently, we calculate
\[ \phi(\lambda, \frac{\lambda}{K}) = \frac{1-\frac{\lambda}{K}}{1-\lambda},\]
and, we see from \eqref{alphabetarelation} that $\lambda/K = \mu^{\frac{1}{p}}\left(\frac{\alpha}{\beta}\right)^{\frac{1}{p'}} < 1$. Since $K \le 1$, we know then that
\[0 \le \phi(\lambda, \frac{\lambda}{K}) \le 1.\]
Consequently,
\begin{equation}
\frac{\avg{w_{a,\nu}^p}^{1/p}_{[\alpha,\beta]}}{\avg{w_{a,\nu}}_{[\alpha,\beta]}} = \theta K \phi(\lambda, \mu) \le \theta = \frac{1+\nu}{(p\nu + 1)^{1/p}},
\end{equation}
which is our desired result. This bound is achieved at $\alpha = 0, \beta = a$. We note here that this bound doesn't depend on $a$. Consequently, we don't need to treat the case of $\beta < a$ separately, since in that case $a$ is simply a multiplicative constant which doesn't affect the norm. The case of $\beta = a$ (which we have considered) is sufficient to cover this.
\end{proof}

\begin{lem}\label{RHinftynorm}
The $RH_{\infty}$ constant for the weight
\[w_{c,a,\nu}(t) = \begin{cases} c\left(\frac{t}{a}\right)^{\nu}& \text{ if } 0 \le t \le a \\
c& \text{ if } a \le t \le 1, \\ \end{cases} \]
with $0 < a \le 1$, $c \ne 0$ and $\nu > 0$ is $\nu + 1$.
\end{lem}
\begin{proof}
We seek the supremum of
\begin{equation}\label{RHinftyratio}
\frac{\esup_J w}{\avg{w}_J}
\end{equation}
over all subintervals $J \subset I:=[0,1]$. We again notice that the value of $c$ is immaterial, as \eqref{RHinftyratio} doesn't change when $w$ is multiplied by a constant. Thus, we take $c=1$. We work with intervals $J = [\alpha, \beta]$ and first prove that it is sufficient to consider $0 \le \alpha < \beta \le a$. Clearly, if $\alpha \ge a$, \eqref{RHinftyratio} is equal to one, which is not maximal; so we only consider $\alpha < a$. If $\beta > a$, then \eqref{RHinftyratio} is equal to
\[\frac{(\beta - \alpha)}{\int_{\alpha}^a (t/a)^{\nu} \, dt + (\beta - a)},\] which is maximized when $\beta = a$.

Consequently, we consider $0 \le \alpha < \beta \le a$. With this, \eqref{RHinftyratio} becomes
\[\frac{(\beta/a)^{\nu}(\beta - \alpha)}{a^{-\nu}\int_{\alpha}^{\beta} t^{\nu} \, dt} = (\nu + 1)\left(\frac{1-\frac{\alpha}{\beta}}{1-(\frac{\alpha}{\beta})^{\nu+1}}\right).\]
Let $y = \alpha/\beta$ and consider the function $y \to \frac{1-y}{1-y^{\nu+1}}$. On the interval $y \in [0,1)$, this function is maximized at $y = 0$ and has a value of $1$. Hence, $\alpha = 0$. We then see that \eqref{RHinftyratio} is equal to $\nu + 1$ on the interval $J=[0, \beta]$ for any $\beta \le a$, and this is our desired constant.
\end{proof}

\section{Acknowledgement}
We would like to thank Professor Sasha Volberg for his stimulating lectures on the Bellman function technique and for the helpful references and guidance he provided us with.

\end{document}